\documentclass[a4paper]{article}
\usepackage{amssymb}
\usepackage{amsthm}
\usepackage{latexsym}
\usepackage{amsmath}
\usepackage{color}

\usepackage{comment}
\excludecomment{discuss}
\excludecomment{discuss2}
\includecomment{discuss3}

\usepackage[margin=1in]{geometry}	
\usepackage{setspace}			

\newif\ifcol
\coltrue
\ifcol
\newcommand{\colorr}{\color[rgb]{0.8,0,0}}
\else
\newcommand{\colorr}{\color{black}}
\fi

\newtheorem{definition}{Definition}[section]
\newtheorem{lemma}{Lemma}[section]
\newtheorem{proposition}{Proposition}[section]
\newtheorem{theorem}{Theorem}[section]
\newtheorem{remark}{Remark}[section]	

\newtheorem{corollary}{Corollary}[section]

\makeatletter
\@addtoreset{equation}{section}
\makeatother

\newcommand{\mca}{\mathcal{A}}
\newcommand{\mcb}{\mathcal{B}}

\newcommand{\mcd}{\mathcal{D}}
\newcommand{\mce}{\mathcal{E}}
\newcommand{\mcf}{\mathcal{F}}
\newcommand{\mcg}{\mathcal{G}}
\newcommand{\mci}{\mathcal{I}}
\newcommand{\mcl}{\mathcal{L}}
\newcommand{\mcm}{\mathcal{M}}
\newcommand{\mcn}{\mathcal{N}}

\newcommand{\mcq}{\mathcal{Q}}
\newcommand{\mcr}{\mathcal{R}}

\newcommand{\mcx}{\mathcal{X}}
\newcommand{\mcy}{\mathcal{Y}}

\newcommand{\mbbn}{\mathbb{N}}

\newcommand{\mbbr}{\mathbb{R}}

\newcommand{\mbbz}{\mathbb{Z}}
\newcommand{\mfa}{\mathfrak{A}}
\newcommand{\mfb}{\mathfrak{B}}
\newcommand{\mfg}{\mathfrak{G}}
\newcommand{\mfi}{\mathfrak{I}}
\newcommand{\mfs}{\mathfrak{S}}
\newcommand{\F}{{\bf F}}

\newcommand{\EQ}[1]{\begin{equation}{#1}\end{equation}}
\newcommand{\EQQ}[1]{\begin{equation}{#1 \nonumber}\end{equation}}
\newcommand{\EQN}[1]{\begin{equation}\begin{split}{#1}\end{split}\end{equation}}
\newcommand{\EQNN}[1]{\begin{equation}\begin{split}{#1 \nonumber}\end{split}\end{equation}}
\newcommand{\MAT}[2]{\left(\begin{array}{#1} #2 \end{array} \right)} 
\newcommand{\DIV}[2]{\left\{\begin{array}{#1} #2 \end{array} \right.} 
\newcommand{\ITEM}[1]{\begin{itemize} #1 \end{itemize}}

\newcommand{\PT}{\partial_\theta}
\newcommand{\TOD}{\overset{d}\to}
\newcommand{\TOP}{\overset{P}\to}
\newcommand{\BDG}{Burkholder-Davis-Gundy inequality} 
\newcommand{\CS}{Cauchy-Schwartz inequality} 

\newcommand{\SI}{S_i^{n,l}}
\newcommand{\SIm}{S_{i-1}^{n,l}}
\newcommand{\HSN}{\hat{\sigma}_n}
\newcommand{\HTN}{\hat{\theta}_n}
\newcommand{\INVS}{S_n^{-1}(\sigma)}
\newcommand{\INVSz}{S_n^{-1}(\sigma_0)}
\newcommand{\PS}{\partial_\sigma}
\newcommand{\SL}{S_{n,0}^{(k)}}
\newcommand{\PSINVS}{\PS S_{n,0}^{-1}}
\newcommand{\SUMK}{\sum_{k=1}^{q_n}}
\newcommand{\SUML}{\sum_{k=1}^{L_n}}
\newcommand{\SUMLL}{\sum_{l=1}^2}
\newcommand{\SUMP}{\sum_{p=0}^\infty}
\newcommand{\SUMPP}{\sum_{p=1}^\infty}
\newcommand{\RITD}{\tilde{\mcd}^{-1/2}}

\newcommand{\NJ}{\mcn^{n,l}}
\newcommand{\BN}{\bar{\mcn}}
\newcommand{\TSIG}{\tilde{\Sigma}}
\newcommand{\TMCD}{\tilde{\mcd}}
\newcommand{\TILG}{\tilde{G}}
\newcommand{\EK}{\mce_{(k)}}

\begin{document}

\title{Asymptotically efficient estimation for diffusion processes with nonsynchronous observations}
\author{Teppei Ogihara$^{*,\dagger}$\\
$*$
\begin{small}Graduate School of Information Science and Technology, University of Tokyo, \end{small}\\
\begin{small}7-3-1 Hongo, Bunkyo-ku, Tokyo 113--8656, Japan\end{small}\\
\begin{small}E-mail: ogihara@mist.i.u-tokyo.ac.jp\end{small} \\
$\dagger$
\begin{small}The Institute of Statistical Mathematics\end{small}\\
}
\maketitle


\noindent
{\bf Abstract.}
We study maximum-likelihood-type estimation for diffusion processes when the coefficients are nonrandom and observation occurs in nonsynchronous manner.
The problem of nonsynchronous observations is important when we consider the analysis of high-frequency data in a financial market.
Constructing a quasi-likelihood function to define the estimator, we adaptively estimate the parameter for the diffusion part and the drift part.
We consider the asymptotic theory when the terminal time point $T_n$ and the observation frequency goes to infinity,
and show the consistency and the asymptotic normality of the estimator.
Moreover, we show local asymptotic normality for the statistical model, and asymptotic efficiency of the estimator as a consequence.
To show the asymptotic properties of the maximum-likelihood-type estimator, we need to control the asymptotic behaviors of some functionals of the sampling scheme. 
Though it is difficult to directly control those in general, we study tractable sufficient conditions when the sampling scheme is generated by mixing processes. 

~

\noindent
{\bf Keywords.}
asymptotic efficiency; diffusion processes; local asymptotic normality; maximum-likelihood-type estimation; nonsynchronous observations

\begin{discuss}
{\colorr
Lemma: \ref{}, {} ; \ref{EX-XJ-lemma}, {EX-XJ-lemma} ; \ref{Sn-lemma}, {Sn-lemma} ; \ref{G-est-lemma}, {G-est-lemma} 

Prop: \ref{}, {} ; \ref{A3-suff-prop}, {A3-suff-prop} ;

Eq: \ref{}, {} ; \ref{XJz-est-eq}, {XJz-est-eq} ; \ref{InvS-eq}, {InvS-eq} ; \ref{H1n-eq}, {H1n-eq} ; \ref{Ln-cond}, {Ln-cond} ; \ref{cnhn-cond}, {cnhn-cond} ; \ref{H2n-eq}, {H2n-eq} ; \ref{HSN-conv}, {HSN-conv} ; \ref{drift-consis-eq3}, {drift-consis-eq3} ; 
\ref{mixing-condition}, {mixing-condition} ; \ref{A3-average-cond}, {A3-average-cond} ; \ref{bar-p-prob}, {bar-p-prob} ; 
\ref{uniform-conti-eq}, {uniform-conti-eq} ; \ref{H1n-est-eq1}, {H1n-est-eq1} ;

式番と仮定と命題の依存関係をグラフで書くと修正する時の影響がすぐわかって効率的。式番が変わることも考えてtex内でグラフを書く。\\

期限が決まっているからまず最低限を作り上げる．LANはとりあえずなくても\\
ジャンプなしのMLEとLANもみて，ジャンプある場合もおそらくeffとコメント\\
余裕があったら：
\ITEM{\item 数値でちゃんと推定できることを見る．refresh sampling + Shimizu and Yoshidaと比較か．}
}
\end{discuss}

\begin{discuss}
{\colorr ~\\ \\ ＜＜確認のための重要な数式などのリスト＞＞\\
※式やLemmaのラベルはrefにかえて，Theorem環境も外す．定理番号がセクション毎なら環境は外さなくてもいいか。元の式を更新したら忘れずに更新する．

(\ref{bar-p-prob})式：
\EQ{P(\bar{\rho}_n<1)\to 1}

(\ref{InvS-eq})式：
\EQQ{ \INVS=\tilde{\mcd}^{-1/2}\sum_{p=0}^\infty \MAT{cc}{(\TILG\TILG^\top)^p & -(\TILG\TILG^\top)^p\TILG \\ -(\TILG^\top \TILG)^p\TILG^\top & (\TILG^\top \TILG)^p}\tilde{\mcd}^{-1/2}.
}

(\ref{H1n-eq})式：
\EQQ{H_n^1(\sigma)-H_n^1(\sigma_0)=-\frac{1}{2}\Delta X^c(\INVS-\INVSz)\Delta X^c-\frac{1}{2}\log\frac{\det S_n(\sigma)}{\det S_n(\sigma_0)}+o_p(\sqrt{n}).}

(\ref{Ln-cond})式：
\EQQ{ (n^2/L_n^3)\vee (L_n/n)\to 0.}

(\ref{HSN-conv})式：
\EQQ{ \sqrt{n}(\HSN-\sigma_0)\TOD N(0,\Gamma_1^{-1}).}

(\ref{V-est})式：
\EQQ{ |{\rm Abs}(\Delta V(\theta))|=\sqrt{\sum_i(\Delta_i V(\theta))^2}=\bar{R}_n(\sqrt{n}h_n), }

(\ref{drift-consis-eq3})式：
\EQQ{\bar{X}(\theta)^\top E S_n^{-1}(\HSN)E\bar{X}(\theta) -\Delta X^\top E S_n^{-1}(\HSN)E \Delta X 
= -2\Delta V(\theta)^\top  S_{n,0}^{-1}\Delta X^c-\Delta V(\theta)^\top S_{n,0}^{-1}(2\Delta V(\theta_0) - \Delta V(\theta))+o_p(\sqrt{nh_n}).}

}
\newpage 
\end{discuss}


\section{Introduction}

Given a probability space $(\Omega, \mathcal{F}, P)$ with a right-continuous filtration $\F=\{\mathcal{F}_t\}_{t\geq 0}$,
let $X^{(\alpha)}=\{X_t^{(\alpha)}\}_{t\geq 0}=\{(X^{(\alpha),1}_t,X^{(\alpha),2}_t)\}_{t\geq 0}$ be a two-dimensional $\F$-adapted process satisfying the following stochastic differential equation
\EQ{\label{SDE} dX_t^{(\alpha)}=\mu_t(\theta)dt+b_t(\sigma) dW_t, \quad X_0=x_0,}
where $x_0\in \mbbr^2$,
$\{W_t\}_{0\leq t\leq T}$ is a two-dimensional standard $\F$-Wiener process,
$\{\mu_t(\theta)\}_{t\geq 0}$ and $\{b_t(\sigma)\}_{t\geq 0}$ are deterministic functions with values in $\mathbb{R}^{2}$ and $\mathbb{R}^{2\times 2}$, respectively,
$\alpha=(\sigma,\theta)$, $\sigma \in \Theta_1$, $\theta\in \Theta_2$, and $\Theta_1$ and $\Theta_2$ are bounded open subsets of $\mathbb{R}^{d_1}$ and $\mbbr^{d_2}$, respectively.
Let $\alpha_0=(\sigma_0,\theta_0)\in \Theta_1\times \Theta_2$ be the true value, and let $X_t=(X_t^1,X_t^2)=X_t^{(\alpha_0)}$.
We consider estimation of $\alpha_0$ when $X$ is observed with nonsynchronous manner, that is, observation times of $X^1$ and $X^2$ are different each other.

The problem of nonsynchronous observations appears in the analysis of high-frequency financial data.
If we analyze the intra-day stock price data, we observe stock price when a new transaction or a new order arrived.
Then, the observation times are different for different stocks, and hence, we cannot avoid the problem of nonsynchronous observations.
Statistical analysis with such data is much more complicated compared to the analysis with synchronous data.
Parametric estimation for diffusion processes with synchronous and equidistant observations have been analyzed through quasi-maximum likelihood methods in Florens-Zmirou~\cite{flo89}, Yoshida~\cite{yos92, yos11}, Kessler~\cite{kes97}, and Uchida and Yoshida~\cite{uch-yos12}.
Related to the estimation problem for nonsynchronously observed diffusion processes, estimators for the quadratic covariation have been actively studied. Hayashi and Yoshida~\cite{hay-yos05, hay-yos08, hay-yos11} and Malliavin and Mancino~\cite{mal-man02, mal-man09} have independently constructed consistent estimators under nonsynchronous observations. 
There are also studies of covariation estimation under the simultaneous presence of microstructure noise and nonsynchronous observations (Barndorff-Nielsen et al.~\cite{bar-etal11}, Christensen, Kinnebrock, and Podolskij~\cite{chr-etal10}, Bibinger et al.~\cite{bib-etal14}, and so on). 
For parametric estimation with nonsynchronous observations, Ogihara and Yoshida~\cite{ogi-yos14} have constructed maximum-likelihood-type and Bayes-type estimators and have shown the consistency and the asymptotic mixed normality of the estimators when the terminal time point $T_n$ is fixed and the observation frequency goes to infinity. Ogihara~\cite{ogi15} have shown local asymototic mixed normality for the model in~\cite{ogi-yos14}, and the maximum-likelihood-type and Bayes-type estimators have been shown to be asymptotically efficient.
On the other hand, we need to consider asymptotic theory that the terminal time point $T_n$ goes to infinity to consistently estimate the parameter $\theta$ in the drift term. To the best of the author's knowledge, there are no study of the asymptotic theory of parametric estimation for nonsynchronously observed diffusion processes when $T_n\to\infty$.

In this work, we consider the asymptotic theory for nonsynchronously observed diffusion processes when $T_n\to\infty$, and construct maximum-likelihood-type estimators for the parameter $\sigma$ in the diffusion part and the parameter $\theta$ in the drift part.
We show the consistency and the asymptotic normality of the estimators. Moreover, we show local asymptotic normality of the statistical model, and we obtain asymptotic efficiency of our estimator as a consequence.
Our estimator is constructed based on the quasi-likelihood function that is similarly defined to the one in~\cite{ogi-yos14} though we need some modification to deal with the drift part.
To investigate asymptotic theory for the maximum-likelihood-type estimator, we need to specify the limit of the quasi-likelihood function.
Then, we need to assume some conditions for the asymptotic behavior of the sampling scheme.
In~\cite{ogi-yos14}, for a matrix 
$$G=\bigg\{\frac{(S_i^{n,1}\wedge S_j^{n,2}-S_{j-1}^{n,2} \vee S_{i-1}^{n,1})\vee 0}{|S_i^{n,1}-S_{i-1}^{n,1}|^{1/2}|S_j^{n,2}-S_{j-1}^{n,2}|^{1/2}}\bigg\}_{i,j}$$
generated by the sampling scheme, the existence of the probability limit of $n^{-1}{\rm tr}((GG^\top)^p) \ (p\in\mbbz_+)$ is required,
where $(S_i^{n,l})_i$ is observation times of $X^l$ and $\top$ denotes transpose of a matrix.
Since we consider the different asymptotics, the asymptotic behavior of the quasi-likelihood function is different from that in~\cite{ogi-yos14}.
We also need to consider estimation for the drift parameter $\theta$. Then, we need other assumptions for the asymptotic behavior of the sampling scheme (Assumption (A5)).
Though these conditions for the sampling scheme is difficult to check directly, we study tractable sufficient conditions in Section~\ref{suff-cond-subsection}.

As seen in~\cite{ogi-yos14}, the quasi-likelihood analysis for nonsynchronously observed diffusion processes become much more complicated compared to synchronous observations.
In this work, estimation for the drift parameter $\theta$ is added, and hence, we consider nonrandom drift and diffusion coefficients to avoid overcomplication. 
For general diffusion processes with the random drift and diffusion coefficients, we need to set predictable coefficients to use the matingale theory. However, the quasi-likelihood function loses a Markov property with nonsynchronous observations and the coefficients in the quasi-likelihood function contains randomness of future time.
Then, we need to approximate the coefficients by predictable functions. This operation is particularly complicated. 
Moreover, approximating the true likelihood function by the quasi-likelihood function is much more difficult problem when we show local asymtotic normality and asymptotic efficiency of the estimators.
Therefore, we left asymptotic theory under general random drift and diffusion coefficients as a future work.

The rest of this paper is organized as follows.
In Section~\ref{main-results-section}, we introduce our model settings and the assumptions for main results.
Our estimator is constructed in Section~\ref{setting-subsection}, and the asymptotic normality of the estimator is given in Section~\ref{main-subsection}. Section~\ref{LAN-subsection} deal with local asymptotic normality of our model and asymptotic efficiency of the estimator. Tractable sufficient conditions for the assumptions of the sampling scheme are given in Section~\ref{suff-cond-subsection}.
Section~\ref{proofs-section} contains the proofs of main results. Section~\ref{HSN-consis-subsection} is for the consistency of the estimator for $\sigma$, Section~\ref{HSN-clt-subsection} is for the asymptotic normality of the estimator for $\sigma$, Section~\ref{HTN-consis-subsection} is for the consistency of the estimator for $\theta$, and Section~\ref{HTN-clt-subsection} is for the asymptotic normality of the estimator for $\theta$.
Other proofs are collected in Section~\ref{suff-cond-proof-subsection}.

\section{Main results}\label{main-results-section}

\subsection{Settings}\label{setting-subsection}

For $l\in\{1,2\}$, let the observation times $\{S_i^{n,l}\}_{i=0}^{M_l}$ be strictly increasing random times with respect to $i$, and satisfy $S_0^{n,l}=0$ and $S_{M_l}^{n,l}=nh_n$, where $M_l$ is a random positive integer depending on $n$.
We assume that $\{S_i^{n,l}\}_{0\leq i\leq M_l, l=1,2}$ is independent of $\mathcal{F}_T$ and $\alpha$.
We consider nonsynchronous observations of $X$, that is, we observe $\{S_i^{n,l}\}_{0\leq i\leq M_l, l=1,2}$ and $\{X^l_{S^{n,l}_i}\}_{0\leq i\leq M_l, l=1,2}$.


We denote by $\lVert\cdot \rVert$ the operator norm of a matrix, and by $\top$ the transpose operator for a matrix or a vector.
We often regard a $p$-dimensional vector $v$ as a $p\times 1$ matrix.
For $j\in\mbbn$ and a vector $\kappa=(\kappa_1,\cdots, \kappa_j)$,
we denote $\partial_{\kappa}^k=(\frac{\partial^k}{\partial\kappa_{i_1}\cdots \partial\kappa_{i_k}})_{i_1,\cdots i_k=1}^j$.
For a set $A$ in a topological space, let ${\rm clos}(A)$ denote the closure of $A$.
For a matrix $A$, $[A]_{ij}$ denotes its $(i,j)$ element.
For a vector $v=(v_j)_{j=1}^K$, ${\rm diag}(v)$ denotes a $k\times k$ diagonal matrix with elements $[{\rm diag}(v)]_{jj}=v_j$.

Let $M=M_1+M_2$. For $1\leq i\leq M$, let 
\EQQ{\varphi(i)=\DIV{ll}{i & {\rm if} \ i\leq M_1\\ i-M_1 & {\rm if} \ i>M_1}
\quad \psi(i)=\DIV{ll}{1 & {\rm if} \ i\leq M_1\\ 2 & {\rm if} \ i>M_1}}
For a two-dimensional stochastic process $(U_t)_{t\geq 0}=((U_t^1,U_t^2))_{t\geq 0}$,
let $\Delta_i^l U=U^l_{S^{n,l}_i}-U^l_{S^{n,l}_{i-1}}$ , and let $\Delta^l U=(\Delta_i^l U)_{1\leq i\leq M_l}$ and $\Delta_i U=\Delta_{\varphi(i)}^{\psi(i)} U$ for $1\leq i\leq M$. Let $\Delta U=((\Delta^1 U)^\top, (\Delta^2 U)^\top)^\top$.
Let $|K|=b-a$ for an interval $K=(a,b]$.
Let $I_i^l=(S_{i-1}^{n,l},S_i^{n,l}]$ for $1\leq i\leq M_l$, and let $I_i=I_{\varphi(i)}^{\psi(i)}$ for $1\leq i\leq M$.
We denote a unit matrix of size $k$ by $\mathcal{E}_k$.

Let $\TSIG_i^l(\sigma)=\int_{I_i^l}[b_tb_t^\top(\sigma)]_{ll}dt$ and $\TSIG_{i,j}^{1,2}(\sigma)=\int_{I_i^1\cap I_j^2}[b_tb_t^\top(\sigma)]_{12}dt$.
By setting 
$\TMCD={\rm diag}(\{\TSIG_i\}_{1\leq i\leq M})$,
\EQQ{G=\bigg\{\frac{|I_i^1\cap I_j^2|}{|I_i^1|^{1/2}|I_j^2|^{1/2}}\bigg\}_{1\leq i\leq M_1, 1\leq j\leq M_2},
\quad \rho_{ij}(\sigma)=\frac{\TSIG_{i,j}^{1,2}}{\sqrt{\TSIG_i^1}\sqrt{\TSIG_j^2}}(\sigma),
\quad \TILG(\sigma)=\{\rho_{ij}(\sigma)[G]_{ij}\}_{1\leq i\leq M_1, 1\leq j\leq M_2},}
we can calculate the covariance matrix of $\Delta X$ as 
\EQQ{S_n(\sigma)=\TMCD^{1/2}\MAT{cc}{\mce_{M_1} &  \TILG(\sigma) \\ \TILG^\top(\sigma) & \mce_{M_2}}\TMCD^{1/2}.}

As we will see later, we can ignore the drift term when we consider estimation of $\sigma$ because the drift term converges to zero very fast.
Therefore, we first construct an estimator for $\sigma$, and then construct an estimator for $\theta$.
Such adaptive estimation can speed up the calculation.

We define the quasi-likelihood function $H_n^1(\sigma)$ for $\sigma$ as follows.
\EQNN{H_n^1(\sigma)&=-\frac{1}{2}\Delta X^\top S_n^{-1}(\sigma)\Delta X-\frac{1}{2}\log \det S_n(\sigma).}
Then, the maximum-likelihood-type estimator for $\sigma$ is defined by
\EQQ{\HSN\in {\rm argmax}_{\sigma\in {\rm clos}(\Theta_1)} H_n^1(\sigma).}

We consider estimation for $\theta$ in the next.
Let $V(\theta)=(V_t(\theta))_{t\geq 0}$ be a two-dimensional stochastic process defined by $V_t(\theta)=(\int_0^t\mu^1_s(\theta)^\top ds,\int_0^t\mu^2_s(\theta)^\top ds)^\top$.
Let $\bar{X}(\theta)=\Delta X-\Delta V(\theta)$.
We define the quasi-likelihood function $H_n^2(\theta)$ for $\theta$ as follows.
\EQQ{H_n^2(\theta)=-\frac{1}{2}\bar{X}(\theta)^\top S_n^{-1}(\HSN)\bar{X}(\theta).}
Then, the maximum-likelihood-type estimator for $\theta$ is defined by
\EQQ{\HTN\in {\rm argmax}_{\theta\in {\rm clos}(\Theta_2)} H_n^2(\theta).}

The quasi-(log-)likelihood function $H_n^1$ is defined in the same way as that in~\cite{ogi-yos14}.
Since $\Delta X$ follows normal distribution, we can construct such a Gaussian quasi-likelihood function even for the nonsynchronous data.
When the coefficients are random, though the distribution of $\Delta X$ is not Gaussian, such Gaussian-type quasi-likelihood function is still valid due to the local Gaussian property of diffusion processes.
The Gaussian mean that comes from the drift part is ignored when we construct the quasi-likelihood $H_n^1$. 
When we estimate the parameter $\theta$ for the drift part, we substruct the mean in $\bar{X}(\theta)$ to construct the quasi-likelihood function $H_n^2$. 
Since the effect of the drift term on the estimation of $\sigma$ is small, it works well to estimate $\sigma$ in this way and then plug in $\HTN$ to $S_n$
to construct the estimator for $\theta$. 
Thus, we can speed up the calculation by separating the estimation for $\sigma$ and $\theta$.

\begin{remark}
$H_n^1(\sigma)$ and $H_n^2(\theta)$ are well-defined only if $\det S_n(\sigma)>0$ and $\det S_n(\HSN)>0$, respectively.
For the covariance matrix $S_n$ of nonsynchronous observations $\Delta X$, it is not trivial to check these conditions.
Proposition 1 in Section 2 of~\cite{ogi-yos14} shows that these conditions are satisfied if $b_t(\sigma)$ is continuous on $[0,\infty)\times {\rm clos}(\Theta_1)$ and $\inf_{t,\sigma} \det (b_tb_t^\top(\sigma))>0$. We assume such conditions in our setting (Assumption (A1) in Section 2.2).
\end{remark}

\subsection{Asymptotic normality of the estimator}\label{main-subsection}

In this section, we state the assumptions of our main results, and state the asymtotic normality of the estimator.

For $m\in \mbbn$, an open subset $U\subset \mbbr^m$ is said to admit Sobolev's inequality if for any $p>m$, there exists a positve constant $C$ depending $U$ and $p$ such that $\sup_{x\in U}|u(x)|\leq C\sum_{k=0,1}(\int |\partial_x^k u(x)|^p)^{1/p}$
for any $u\in C^1(U)$. This is the case when $U$ has a Lipschitz boundary.
We assume that $\Theta$, $\Theta_1$, and $\Theta_2$ admit Sobolev's inequality.

Let $\Sigma_t(\sigma)=b_tb_t^\top(\sigma)$, and let
\EQQ{\rho_t(\sigma)=\frac{[\Sigma_t]_{12}}{[\Sigma_t]_{11}^{1/2}[\Sigma_t]_{22}^{1/2}}(\sigma), \quad B_{l,t}(\sigma)=\frac{[\Sigma_t(\sigma_0)]_{ll}}{[\Sigma_t(\sigma)]_{ll}}.}
Let $\rho_{t,0}=\rho_t(\sigma_0)$ and $r_n=\max_{i,l}|I_i^l|$.
Let $\mfs$ be the set of all partitions $(s_k)_{k=0}^\infty$ of $[0,\infty)$ satisfying $\sup_{k\geq 1}|s_k-s_{k-1}|\leq 1$ and $\inf_{k\geq 1}|s_k-s_{k-1}|>0$.
For $(s_k)_{k=0}^\infty\in \mfs$, let $M_{l,k}=\#\{i;\sup I_i^l\in (s_{k-1},s_k]\}$ and $q_n=\max\{k;s_k\leq nh_n\}$,
and let $\EK^l$ be an $M_l\times M_l$ matrix satisfying
$[\EK^l]_{ij}=1$ if $i=j$ and $\sup I_i^l\in (s_{k-1},s_k]$, and otherwise $[\EK^l]_{ij}=0$.
\begin{discuss}
{\colorr 分離を示すときに$(G_kG_k^\top)^p$と$\mce_{(k)}(GG^\top)^p$の誤差評価で等間隔$s_k$だけでなく，any partitionが必要}
\end{discuss}

\begin{description}
\item[Assumption (A1).] There exist positive constants $c_1$ and $c_2$ such that $c_1\mce_2 \leq \Sigma_t(\sigma) \leq c_2\mce_2$ for any $t\in 
[0,\infty)$ and $\sigma \in \Theta_1$. For $k\in\{0,1,2,3,4\}$, $\PT^k\mu_t(\theta)$ and $\PS^kb_t(\sigma)$ exist and are continuous with respect to $(t,\sigma,\theta)$ on $[0,\infty)\times {\rm clos}(\Theta_1)\times {\rm clos}(\Theta_2)$.
For any $\epsilon>0$, there exist $\delta>0$ and $K>0$ such that 
\EQQ{|\PT^k\mu_t(\theta)|+|\PS^kb_t(\sigma)|\leq K, \quad |\PT^k\mu_t(\theta)-\PT^k\mu_s(\theta)|+|\PS^kb_t(\sigma)-\PS^kb_s(\sigma)|\leq \epsilon}
for any $k\in\{0,1,2,3,4\}$, $\sigma \in \Theta_1$, $\theta\in \Theta_2$, and $t,s\geq 0$ satisfying $|t-s|<\delta$.

\begin{discuss}
{\colorr (\ref{bar-p-prob})や$H_n^1,H_n^2$の極限の存在のために$\mu_t,b_t$の$t$に関する一様連続性を仮定する}
\end{discuss}
\item[Assumption (A2).] $r_n\TOP 0$ as $n\to\infty$. 
\item[Assumption (A3).] 
For any $l\in\{1,2\}$, $i_1\in \mbbz_+$, $i_2\in \{0,1\}$, $i_3\in \{0,1,2,3,4\}$, $k_1,k_2\in \{0,1,2\}$ satisfying $k_1+k_2=2$, and any polynomial function $F(x_1,\cdots, x_{14})$ of degree equal to or less than $4$, 
there exist continuous functions $\Phi_{i_1,i_2}^{1,F}(\sigma)$, $\Phi_{l,i_3}^2(\sigma)$ and $\Phi^{3,k_1,k_2}_{i_1,i_3}(\theta)$ on ${\rm clos}(\Theta_1)$ and ${\rm clos}(\Theta_2)$ such that

\EQQ{\frac{1}{T}\int_0^TF((\PS^k B_{l,t}(\sigma))_{0\leq k\leq 4, l=1,2}, (\PS^{k'} \rho_t(\sigma))_{k'=1}^4)\rho_t(\sigma)^{i_1}\rho_{t,0}^{i_2} dt
\to \Phi_{i_1,i_2}^{1,F}(\sigma),}
\EQQ{\frac{1}{T}\int_0^T \PS^{i_3}\log B_{l,t}(\sigma)dt \to \Phi_{l,i_3}^2(\sigma), 
\quad \frac{1}{T}\int_0^T \PT^{i_3}(\phi_{1,t}^{k_1}\phi_{2,t}^{k_2})(\theta)\rho_{t,0}^{i_1}dt \to \Phi^{3,k_1,k_2}_{i_1,i_3}(\theta)
}
as $T\to \infty$ for $\sigma\in{\rm clos}(\Theta_1)$, $\theta\in {\rm clos}(\Theta_2)$,
where $\phi_{l,t}(\theta)=[\Sigma_t(\sigma_0)]_{ll}^{-1/2}(\mu_t^l(\theta)-\mu_t^l(\theta_0))$.
\end{description}
Assumption (A1) and the Ascoli--Arzel\`a theorem yield that the convergences in (A3) can be replaced by uniform convergence with respect to $\sigma$ and $\theta$.
Assumption (A3) is satisfied if $\mu_t(\theta)$ and $b_t(\sigma)$ are independent of $t$, or are periodic functions with respect to $t$ having a common period (when the period does not depend on $\sigma$ nor $\theta$).
\begin{discuss}
{\colorr さらに微分も一様収束することから極限関数が微分できることもわかるか}
\end{discuss}

Let $\mfi_l=(|I_i^l|^{1/2})_{i=1}^{M_l}$.
\begin{description}
\item[Assumption (A4).] 
There exist positive constants $a_0^1$ and $a_0^2$ such that $\{h_nM_{l,q_n+1}\}_{n=1}^\infty$ is $P$-tight and
\EQQ{\max_{1\leq k\leq q_n} |h_nM_{l,k}-a_0^l(s_k-s_{k-1})|\TOP 0}
for $l\in \{1,2\}$ and any partition $(s_k)_{k=0}^\infty \in\mfs$. 
Moreover, for any $p\in \mbbn$, there exists a nonnegative constant $a_p^1$ such that 
\EQQ{\max_{1\leq k\leq q_n}|h_n{\rm tr}(\EK^1(GG^\top)^p)-a_p^1(s_k-s_{k-1})|\TOP 0}
as $n\to\infty$ for any partition $(s_k)_{k=0}^\infty \in\mfs$.
\begin{discuss}
{\colorr Proposition~\ref{H1n-conv-prop}を成り立たせるには(A3)で４階微分まで必要．極限の形が必要なのは二階までで三階以降なくても示せるが，証明がかわるから仮定しておく}
\end{discuss}
\item[Assumption (A5).] For $p\in \mbbz_+$, there exist nonnegative constants $f_p^{1,1}$, $f_p^{1,2}$, and $f_p^{2,2}$ such that 
\EQNN{\max_{1\leq k\leq q_n}|\mfi_1\EK^1(GG^\top)^p\mfi_1-f_p^{1,1}(s_k-s_{k-1})|&\TOP 0, \\
\max_{1\leq k\leq q_n}|\mfi_1\EK^1(GG^\top)^pG\mfi_2-f_p^{1,2}(s_k-s_{k-1})|&\TOP 0, \\
\max_{1\leq k\leq q_n}|\mfi_2\EK^2(G^\top G)^p\mfi_2-f_p^{2,2}(s_k-s_{k-1})|&\TOP 0
}
as $n\to\infty$ for any partition $(s_k)_{k=0}^\infty \in\mfs$.
\end{description}
\begin{discuss}
{\colorr (A5)では$k=q_n+1$のtightnessはすぐ得られる}
\end{discuss}
Assumption (A4) corresponds to [A3$'$] in Ogihara and Yoshida~\cite{ogi-yos14}.
The functionals in (A4) and (A5) appear in $H_n^1$ and $H_n^2$, and hence, we cannot specify the limits of $H_n^1$ and $H_n^2$ unless we assume existence of the limits of these functionals.
It is difficult to directly check (A4) and (A5) for general sampling scheme. We study sufficient conditions for these conditions in Section~\ref{suff-cond-subsection}.


\begin{description}
\item[Assumption (A6).] 
The constant $a_1^1$ in (A4) is positive, and there exist positive constants $c_3$ and $c_4$ such that 
\EQNN{\limsup_{T\to \infty}\bigg(\frac{1}{T}\int_0^T\lVert \Sigma_t(\sigma) -\Sigma_t(\sigma_0)\rVert^2dt\bigg)&\geq c_3|\sigma-\sigma_0|^2, \\
\limsup_{T\to\infty}\bigg(\frac{1}{T}\int_0^T|\mu_t(\theta)-\mu_t(\theta_0)|^2dt\bigg)&\geq c_4|\theta-\theta_0|^2
}
for any $\sigma\in {\rm clos}(\Theta_1)$ and $\theta\in{\rm clos}(\Theta_2)$.
\end{description}
Assumption (A6) is necessary to identify the parameter $\sigma$ and $\theta$ from the data.
If $a_1^1=0$, then we have $a_p^1=0$ for any $p\in\mbbn$. This implies that the non-diagonal components of the covariance matrix $S_n$ are negligible in the limit.
Then, we cannot consistently estimate the parameter in $\rho_t(\sigma)$. This is why we need the assumption $a_1^1>0$ (see Proposition~\ref{mcy1-est-prop} and the following discussion to obtain the consistency).

Let $\mca(\rho)=\sum_{p=1}^\infty a_p^1\rho^{2p}$ for $\rho \in (-1,1)$, and let $\PS^kB_{l,t,0}=\PS^kB_{l,t}(\sigma_0)$.
Let 
\EQQ{\gamma_{1,t}=\mca(\rho_{t,0})\bigg(\frac{\PS \rho_{t,0}}{\rho_{t,0}}-\PS B_{1,t,0}-\PS B_{2,t,0}\bigg)^2-\partial_\rho \mca(\rho_{t,0})\frac{(\PS \rho_{t,0})^2}{\rho_{t,0}}
-2\sum_{l=1}^2(a_0^l+\mca(\rho_{t,0}))(\PS B_{l,t,0})^2,}
and let $\Gamma_1=\lim_{T\to\infty}T^{-1}\int_0^T\gamma_{1,t}dt$, which exists under (A1), (A3) and (A4).
Let
\EQQ{\Gamma_2=\lim_{T\to\infty}\frac{1}{T}\int_0^T\sum_{p=0}^\infty \rho_{t,0}^{2p}\bigg\{\sum_{l=1}^2f_p^{ll}(\PT\phi_{l,t})^2(\theta_0)-2\rho_{t,0}f_p^{12}\PT\phi_{1,t}\PT \phi_{2,t}(\theta_0)\bigg\}dt,}
which exists under (A1), (A3) and (A5).
Let $T_n=nh_n$ and
\EQQ{\Gamma=\MAT{cc}{\Gamma_1 & 0 \\ 0 & \Gamma_2}.}

\begin{theorem}\label{main-thm}
Assume (A1)--(A6). Then $\Gamma$ is positive difinite, and
\EQQ{(\sqrt{n}(\HSN-\sigma_0),\sqrt{T_n}(\HTN-\theta_0))\overset{d}\to N(0,\Gamma^{-1})}
as $n\to\infty$.
\end{theorem}

\subsection{Local asymptotic normality}\label{LAN-subsection}

Next, to discuss the optimality of the estimator, we discuss local asymptotic normality of the statistical model.
In this section, local asymptotic normality of our model is shown, and the maximum-likelihood-type estimator is shown to be asymptotically efficient.

Let $\mbbn$ be the set of all positive integers.
 Let $\alpha_0\in \Theta$, $\Theta \subset \mbbr^d$, and $\{P_{\alpha,n}\}_{\alpha\in\Theta}$ be a family of probability measures defined on a measurable space $(\mathcal{X}_n,\mca_n)$ for $n\in\mbbn$, where $\Theta$ is an open subset of $\mathbb{R}^d$.
As usual we shall refer to $dP_{\alpha_2,n}/dP_{\alpha_1,n}$ the derivative of the absolutely continuous component of the measure $P_{\alpha_2,n}$ with respect to measure $P_{\alpha_1,n}$ at the observation $x$ as the likelihood
ratio.
The following definition of local asymptotic normality is Definition 2.1 in Chapter II of Ibragimov and Has'minski{\u\i}~\cite{ibr-has81}.

\begin{definition}
A family $P_{\alpha,n}$ is called locally asymptotically normal (LAN) at point $\alpha_0\in \Theta$ as $n\to\infty$ if for some nondegenerate $d \times d$ matrix $\epsilon_n$ and any $u\in\mbbr^d$, the representation
\EQQ{\log \frac{dP_{\alpha_0+\epsilon_n u,n}}{dP_{\alpha_0,n}}-(u^\top \Delta_n-|u|^2/2)\to 0}
in $P_{\alpha_0,n}$-probability as $n\to\infty$, where
\EQQ{\mcl(\Delta_n|P_{\alpha_0,n})\to N(0,\mce_d)}
as $n\to\infty$.
\end{definition}

\begin{discuss}
{\colorr 今は係数がdeterministicだから解の存在は言える}
\end{discuss}

Let $\Theta=\Theta_1\times \Theta_2$.
For $\alpha\in\Theta$, let $P_{\alpha,n}$ be the probability measure generated by the observation $\{S_i^{n,l}\}_{i,l}$ and $\{X_{S_i^{n,l}}^{(\alpha),l}\}_{i,l}$.
\begin{theorem}\label{LAN-thm}
Assume (A1)--(A6). Then, $\{P_{\alpha,n}\}_{\alpha,n}$ satisfies the LAN property at $\alpha=\alpha_0$ with
\EQQ{\epsilon_n=\MAT{cc}{n^{-1/2}\Gamma_1^{-1/2} & 0 \\ 0 & T_n^{-1/2}\Gamma_2^{-1/2}}.}
\end{theorem}

The proof is left to Section~\ref{suff-cond-proof-subsection}.
Theorem 11.2 in Chapter II of Ibragimov and Has'minski{\u\i}~\cite{ibr-has81}
gives lower bounds of estimation errors for any regular estimator of parameters under the LAN property. 
Then, the optimal asymptotic variance of $\epsilon_n^{-1}(T_n-\alpha_0)$ for regular estimator $T_n$ is $\mce_d$.
Therefore, Theorems~\ref{LAN-thm} ensures that our estimator $(\HSN,\HTN)$ is asymptotically efficient in this sense under the assumptions of the theorem
(we can show that $(\HSN,\HTN)$ is regular by the proof of Theorem~\ref{LAN-thm}, (\ref{joint-conv-eq}), (\ref{H1n-eq}), (\ref{HSN-error-eq}), (\ref{Sn-exp}) and Theorem 2 in~\cite{jeg82}).

\subsection{Sufficient conditions for the assumptions}\label{suff-cond-subsection}

It is not easy to directly check Assumptions (A4) and (A5) for general random sampling scheme. In this section, we study tractable sufficient conditions for these assumptions.
The proofs of the results in this section are left to Section~\ref{suff-cond-proof-subsection}.

Let $q>0$ and $\NJ_t=\sum_{i=1}^{M_l}1_{\{S_i^{n,l}\leq t\}}$.
We consider the following conditions for point process $\mcn_t^{n,l}$.
\begin{description}
\item[Assumption (B1-$q$).] \EQQ{\sup_{n\geq 1}\max_{l\in\{1,2\}}\sup_{0\leq t\leq (n-1)h_n}E[(\NJ_{t+h_n}-\NJ_t)^q]<\infty.}
\item[Assumption (B2-$q$).] \EQQ{\limsup_{u\to\infty} \sup_{n\geq 1}\max_{l\in\{1,2\}}\sup_{0\leq t\leq nh_n-uh_n} u^qP(\NJ_{t+uh_n}-\NJ_t=0)<\infty.}
\end{description}

For example, let $(\BN_t^1, \BN_t^2)$ be two independent homogeneous Poisson processes with positive intensities $\lambda_1$ and $\lambda_2$, respectively,
and $\NJ_t=\BN_{h_n^{-1}t}^l$.
Then (B1-$q$) obviously holds for any $q > 0$. Moreover, (B2-q) holds for any $q > 0$ since
\EQQ{\limsup_{u\to\infty} \sup_{n\geq 1}\max_{l\in\{1,2\}}\sup_{0\leq t\leq nh_n-uh_n} u^qP(\NJ_{t+uh_n}-\NJ_t=0)=\lim_{u\to\infty}u^qe^{-(\lambda_1\wedge \lambda_2)u}=0.}

\begin{discuss}
{\colorr $\nu_p([0,1])=O_p(h_n^{-1})$, $a_p=\lim_{n\to\infty}(n^{-1}E[\nu_p([0,T_n))])$.
$a_P=O(1)$だが，$a_p$は$[0,T_n]$を$1/n$して得られるので時間単位とリンクしておらず$\int a_pdt$は変．→$a_p$を$t$に依存させない．
$\nu_p([0,t))$を各$t$に対して評価する必要はないから問題ない．
}
\end{discuss}

To give sufficient conditions for (A4) and (A5), we consider mixing properties of $\mcn^{n,l}$.
That is, we assume condtions for the following mixing coefficient $\alpha_k^n$.
Let 
\EQQ{\mcg_{i,j}^n=\sigma(\NJ_t-\NJ_s; ih_n\leq s<t\leq jh_n, l=1,2) \quad (0\leq i,j\leq n),}
and let 
\EQQ{\alpha_k^n=0\vee \sup_{1\leq i,j\leq n-1, j-i\geq k}\sup_{A\in \mcg_{0,i}^n}\sup_{B\in \mcg_{j,n}^n}|P(A\cap B)-P(A)P(B)|.}

\begin{proposition}\label{A3-suff-prop}
Assume that (B1-$q$) and (B2-$q$) hold
and that 
\EQ{\label{mixing-condition} \sup_{n\in\mbbn} \sum_{k=0}^\infty (k+1)^q\alpha_k^n<\infty}
for any $q>0$.
Moreover, assume that there exist positive constants $a_0^1$ and $a_0^2$, and a nonnegative constant $a_p^1$ for $p\in \mbbn$ such that
\EQN{\label{A3-average-cond} \max_{1\leq k\leq q_n}|h_nE[M_{l,k}]-a_0^l(s_k-s_{k-1})|&\to 0, \\
\max_{1\leq k\leq q_n}|h_nE[{\rm tr}(\EK^1(GG^\top)^p)]-a_p^1(s_k-s_{k-1})|&\to 0
}
as $n\to\infty$ for $p\in\mbbz_+$, $l\in\{1,2\}$ and any partition $(s_k)_{k=0}^\infty\in\mfs$.
Then, (A4) holds.
\end{proposition}

In the following, let $(\BN_t^l)_{t\geq 0}$ be an exponential $\alpha$-mixing point process for $l\in \{1,2\}$.
Assume that the distribution of $(\BN_{t+t_k}^l-\BN_{t+t_{k-1}}^l)_{1\leq k\leq K, l=1,2}$ does not depend on $t\geq 0$ for any $K\in\mbbn$ and $0\leq t_0<t_1<\cdots < t_K$.

\begin{proposition}\label{A5-suff-cond-prop}
Assume that (B1-$q$) and (B2-$q$) hold
and that (\ref{mixing-condition}) is satisfied for any $q>0$. 
Moreover, assume that there exist nonnegative constants $f_p^{1,1}$, $f_p^{1,2}$, and $f_p^{2,2}$ for $p\in \mbbz_+$ such that
\EQN{\label{A5-average-cond} \max_{1\leq k\leq q_n}|E[\mfi_1\EK^1(GG^\top)^p\mfi_1]-f_p^{1,1}(s_k-s_{k-1})|&\to 0, \\
\max_{1\leq k\leq q_n}|E[\mfi_1\EK^1(GG^\top)^pG\mfi_2]-f_p^{1,2}(s_k-s_{k-1})|&\to 0, \\
\max_{1\leq k\leq q_n}|E[\mfi_2\EK^2(G^\top G)^p\mfi_2]-f_p^{2,2}(s_k-s_{k-1})|&\to 0
}
as $n\to\infty$ for $p\in\mbbz_+$ and any partition $(s_k)_{k=0}^\infty\in\mfs$. Then, (A5) holds.
\end{proposition}

\begin{proposition}\label{a11-suff-cond-prop}
Assume that there exists $q>0$ such that (A4) and (B2-$q$) hold,
$\{\mcn_{t+h_n}^{n,l}-\mcn_t^{n,l}\}_{0\leq t\leq T_n-h_n, l\in \{1,2\}, n\in\mbbn}$ is $P$-tight,
and $\sum_{k=1}^\infty k\alpha_k^n<\infty$. Then, $a_1^1>0$.
\end{proposition}

\begin{lemma}\label{A4A5-suff-cond-prop}
Let $\NJ_t=\BN_{h_n^{-1}t}^l$ for $0\leq t\leq nh_n$ and $l\in\{1,2\}$. Then, (\ref{mixing-condition}) is satisfied for any $q>2$, and there exist constants $a_0^1$, $a_0^2$, and $a_p^1=a_p^2$ for $p\in \mbbn$ such that (\ref{A3-average-cond}) holds true.
Moreover, there exist nonnegative constants $f_p^{1,1}$, $f_p^{1,2}$, and $f_p^{2,2}$ for $p\in\mbbz_+$ such that (\ref{A5-average-cond}) holds.
\end{lemma}

\begin{proposition}[Proposition 8 in~\cite{ogi-yos14}]
Let $q\in \mbbn$. Assume (B2-$(q+1)$). Then, $\sup_nE[h_n^{-q+1}r_n^q]<\infty$. In particular, (A2) holds under (B2-$1$).
\end{proposition}
By the above results, we obtain simple tractable sufficient conditions for the assumptions of the sampling scheme. 
\begin{corollary}
Let $\mcn_t^{n,l}=\BN_{h_n^{-1}t}^l$ for $0\leq t\leq T_n$ and $l\in\{1,2\}$.
Assume that (B1-$q$) and (B2-$q$) hold for any $q>0$.
Then, (A2), (A4) and (A5) hold, and $a_1^1>0$.
\end{corollary}

\section{Proofs}\label{proofs-section}

\subsection{Preliminary results}

For a real number $a$, $[a]$ denotes the maximum integer which is not greater than $a$.
Let $\Pi=\Pi_n=\{S_i^{n,l}\}_{1\leq i\leq M_l, l\in \{1,2\}}$.
We denote $|x|^2=\sum_{i_1,\cdots, i_k}|x_{i_1,\cdots, i_k}|^2$ for $x=\{x_{i_1,\cdots, i_k}\}_{i_1,\cdots, i_k}$ with $k\in\mbbn$.
$C$ denotes generic positive constant whose value may vary depending on context.
We often omit the parameters $\sigma$ and $\theta$ in general functions $f(\sigma)$ and $g(\theta)$.

\begin{discuss}
{\colorr $R_n$はソボレフのパラメータに関するsup評価がやりやすいからちゃんと定義する}
\end{discuss}
For a sequence $p_n$ of positive numbers, let us denote by $\{\bar{R}_n(p_n)\}_{n\in\mathbb{N}}$
a sequence of random variables (which may also depend on $1\leq i\leq M$ and $\alpha\in \Theta$) satisfying
\begin{equation}
\sup_{\alpha,i}E_\Pi[|p_n^{-1}\bar{R}_n(p_n)|^q]^{1/q}<\infty \quad {\rm a.s.} 
\end{equation}
where $E_\Pi[{\bf X}]=E[{\bf X}|\sigma(\Pi_n)]$ for a random variable ${\bf X}$.

\begin{discuss}
{\colorr $\bar{R}_n$を使うとオーダー評価での$\epsilon$も省けるのでよい}
\end{discuss}

Let $\bar{V}=V(\theta_0)$, $\bar{\rho}_n=\sup_\sigma(\max_{i,j}|\rho_{i,j}(\sigma)| \vee \sup_t|\rho_t(\sigma)|)$, and let
\EQQ{\bar{S}=\MAT{cc}{\mce_{M_1} & G \\ G^\top & \mce_{M_2}}.}
Let $\Delta_{i,t}^l U=U^l_{t\wedge S^{n,l}_i}-U^l_{t\wedge S^{n,l}_{i-1}}$, and let $\Delta_{i,t} U=\Delta_{\varphi(i),t}^{\psi(i)} U$ for $t\geq 0$ and a two-dimensional stochastic process $(U_t)_{t\geq 0}=((U_t^1,U_t^2))_{t\geq 0}$.

\begin{lemma}[Lemma 2 in~\cite{ogi-yos14}]\label{G-est-lemma}
$\lVert G\rVert \vee \lVert G^\top \rVert \leq 1$.
\end{lemma}

\begin{lemma}\label{TILG-est-lemma}
$\lVert \TILG\rVert \vee \lVert \TILG^\top \rVert \leq \bar{\rho}_n$.
\end{lemma}

\begin{proof}
Since all the elements of $G$ are nonnegative, we have
\EQNN{\lVert \TILG \rVert^2 &=\sup_{|x|=1}|\TILG x|^2 =\sup_{|x|=1}\sum_i\bigg(\sum_j \rho_{ij}G_{ij}x_j\bigg)^2 \\
&\leq \bar{\rho}_n^2 \sup_{|x|=1}\sum_i\bigg(\sum_j G_{ij}|x_j|\bigg)^2 \leq \bar{\rho}_n^2\lVert G\rVert^2\leq \bar{\rho}_n^2.
}
Since $\lVert \TILG^\top \rVert=\lVert \TILG\rVert$, we obtain the conclusion.
\end{proof}

Let 
$\mcd={\rm diag}(\{|I_i|\}_{i=1}^M)$.
\begin{lemma}\label{Sn-lemma}
\begin{discuss} 
{\colorr {Sn-lemma}} 
\end{discuss}
Assume (A1). Then, there exists a positive constant $C$ such that $\lVert \mcd^{1/2} \PS^kS_n^{-1}(\sigma)\mcd^{1/2} \rVert \leq C(1-\bar{\rho}_n)^{-k-1}$ if $\bar{\rho}_n<1$, and $\lVert \mcd^{-1/2} \PS^kS_n(\sigma)\mcd^{-1/2}\rVert\leq C$ for any $\sigma\in\Theta_1$ and $k\in \{0,1,2,3,4\}$.

\end{lemma}
\begin{proof}

By (A1) and Lemmas~\ref{G-est-lemma} and~\ref{TILG-est-lemma}, we have
\EQQ{\lVert \mcd^{-1/2}\PS^k S_n(\sigma)\mcd^{-1/2}\rVert\leq C\sum_{j=0}^k\bigg\lVert \PS^j\bigg\{\mce_{M}+ \MAT{cc}{0 & \TILG \\ \TILG^\top & 0}\bigg\}\bigg\rVert \leq C.}

Moreover, by (A1) and Lemma~\ref{G-est-lemma}, we have
\EQQ{\lVert \mcd^{1/2}S_n^{-1}\mcd^{1/2}\rVert \leq C\bigg\lVert \bigg(\mce_M+\MAT{cc}{0 & \TILG \\ \TILG^\top & 0}\bigg)^{-1}\bigg\lVert\leq C(1-\bar{\rho}_n)^{-1}}
if $\bar{\rho}_n<1$.

By using the equation $\PS S_n^{-1}=-S_n^{-1}\PS S_n S_n^{-1}$, we obtain
\EQQ{\lVert \mcd^{1/2}\PS S_n^{-1}\mcd^{1/2}\rVert =\lVert \mcd^{1/2}S_n^{-1}\PS S_nS_n^{-1}\mcd^{1/2}\rVert
\leq \lVert \mcd^{1/2}S_n^{-1}\mcd^{1/2}\rVert^2\lVert \mcd^{-1/2}\PS S_n \mcd^{-1/2}\rVert \leq C(1-\bar{\rho}_n)^{-2}
}
if $\bar{\rho}_n<1$.
Similarly, we obtain 
\EQQ{\lVert \mcd^{1/2}\PS^k S_n^{-1}\mcd^{1/2}\rVert \leq C(1-\bar{\rho}_n)^{-k-1}}
if $\bar{\rho}_n<1$ for $k\in \{0,1,2,3,4\}$.

\end{proof}


$\bar{\rho}_n$ is $\Pi_n$-measurable, and
We obtain
\EQ{\label{bar-p-prob} P(\bar{\rho}_n<1)\to 1}
as $n\to\infty$ by (A2) and uniform continuity of $b_t$ and $\det \Sigma_t>0$ under (A1).
Together with Lemma~\ref{G-est-lemma}, we have
\EQN{\label{InvS-eq} \INVS&=\tilde{\mcd}^{-1/2}\sum_{p=0}^\infty (-1)^p\MAT{cc}{0 & \TILG \\ \TILG^\top & 0}^p\tilde{\mcd}^{-1/2} \\
&=\tilde{\mcd}^{-1/2}\sum_{p=0}^\infty \MAT{cc}{(\TILG\TILG^\top)^p & -(\TILG\TILG^\top)^p\TILG \\ -(\TILG^\top \TILG)^p\TILG^\top & (\TILG^\top \TILG)^p}\tilde{\mcd}^{-1/2}.
}

\subsection{Consistency of $\hat{\sigma}_n$}\label{HSN-consis-subsection}

We first show consistency: $\hat{\sigma}_n\TOP \sigma_0$ as $n\to\infty$.
For this purpose, we specify the limit of $H_n^1(\sigma)-H_n^1(\sigma_0)$.

\begin{lemma}
Assume (A1) and (A2). Then 
\EQN{\label{H1n-eq2} \frac{1}{n}\sup_{\sigma\in\Theta_1}\bigg|\PS^k(H_n^1(\sigma)-H_n^1(\sigma_0))
+\frac{1}{2}\PS^k{\rm tr}(\INVS(S_n(\sigma_0)-S_n(\sigma)))+\frac{1}{2}\PS^k\log\frac{\det S_n(\sigma)}{\det S_n(\sigma_0)}\bigg|
\TOP 0
}
as $n\to\infty$ for $k\in \{0,1,2,3\}$.
\end{lemma}

\begin{proof}
Let $X_t^c=\int_0^tb_s(\sigma_0)dW_s$.
By the definition of $H_n^1$, we have
\EQQ{H_n^1(\sigma)-H_n^1(\sigma_0)=-\frac{1}{2}\Delta X^\top (S_n^{-1}(\sigma)-S_n^{-1}(\sigma_0))\Delta X -\frac{1}{2}\log\frac{\det S_n(\sigma)}{\det S_n(\sigma_0)}.}
Since
\EQ{\label{Xc-approx-eq} \Delta X^\top S_n^{-1}(\sigma)\Delta X-(\Delta X^c)^\top S_n^{-1}(\sigma)\Delta X^c
=(\Delta \bar{V})^\top S_n^{-1}(\sigma)(2\Delta X^c+\Delta \bar{V}),}
and
\EQ{\label{DeltaV2} |\mcd^{-1/2}\Delta \bar{V}|^2=\sum_{i,l}|I_i^l|^{-1}|\Delta_i^l \bar{V}|^2\leq Cnh_n,}
\begin{discuss}
{\colorr {DeltaV2}}
\end{discuss}
together with Lemma~\ref{Sn-lemma} and (\ref{bar-p-prob}), we obtain
\EQ{\label{VSV-est} |(\Delta \bar{V})^\top \INVS \Delta \bar{V}|\leq \lVert \mcd^{1/2}\INVS\mcd^{1/2} \rVert |\mcd^{-1/2}\Delta \bar{V}|^2=O_p(nh_n)=o_p(\sqrt{n}).}
Moreover, Lemma~\ref{Sn-lemma}, (\ref{bar-p-prob}), (\ref{DeltaV2}) and the equation $E_\Pi[\Delta X^c(\Delta X^c)^\top]=S_n(\sigma_0)$ yield
\EQ{\label{VSXC-est} E_\Pi[|(\Delta \bar{V})^\top S_n^{-1}(\sigma)\Delta X^c|^2]
=(\Delta \bar{V})^\top S_n^{-1}(\sigma)E_\Pi[\Delta X^c(\Delta X^c)^\top] S_n^{-1}(\sigma)\Delta \bar{V}= O_p(nh_n)=o_p(\sqrt{n}).}

(\ref{Xc-approx-eq}), (\ref{VSV-est}), and (\ref{VSXC-est}) yield
\EQ{\label{H1n-eq} H_n^1(\sigma)-H_n^1(\sigma_0)=-\frac{1}{2}(\Delta X^c)^\top (\INVS-\INVSz)\Delta X^c-\frac{1}{2}\log\frac{\det S_n(\sigma)}{\det S_n(\sigma_0)}+o_p(\sqrt{n}).}

It\^o's formula yields
\EQN{\label{XSX-martingale-est} &(\Delta X^c)^\top\INVS\Delta X^c - {\rm tr}(\INVS S_n(\sigma_0)) \\
&\quad =\sum_{i,j}[\INVS]_{ij}(\Delta_i X^c\Delta_j X^c-[\Sigma_0]_{\psi(i),\psi(j)}|I_i\cap I_j|) \\
&\quad =\sum_{i,j}[\INVS]_{ij}\bigg\{\int_{I_i}\Delta_{j,t} X^cdX^{c,\psi(i)}_t+\int_{I_j}\Delta_{i,t} X^cdX_t^{c,\psi(j)}\bigg\} \\
&\quad =2\sum_{i,j}[\INVS]_{ij}\int_{I_i}\Delta_{j,t} X^cdX_t^{c,\psi(i)},
}
where $X^{c,l}_t$ is $l$-the component of $X_t$.
\begin{discuss}
{\colorr 左辺の$E_\Pi[(\quad )^2]$評価を示せばいいから, $S_i^{n,l}$をdeterministicに代えてよくて伊藤の公式が使える}
\end{discuss}

Since $\langle \Delta_i X^c, \Delta_j X^c\rangle_t=\int_{[0,t)\cap I_i\cap I_j}[\Sigma_t]_{\psi(i),\psi(j)}dt$, together with 
Lemma~\ref{Sn-lemma}, (\ref{bar-p-prob}) and the \BDG, we have
\EQNN{&E_\Pi\bigg[\bigg(\sum_{i,j}[S_n^{-1}(\sigma)]_{ij}\int_{I_i}\Delta_{j,t} X^c dX_t^{c,\psi(i)}\bigg)^q\bigg] \\
&\quad \leq C_q\sum_{l=1}^2E_\Pi\bigg[\bigg(\sum_{\substack{i,j_1,j_2 \\ \psi(i)=l}}[S_n^{-1}(\sigma)]_{i,j_1}[S_n^{-1}(\sigma)]_{i,j_2}
\int_{I_i} \Delta_{j_1,t} X^c \Delta_{j_2,t} X^c[\Sigma_t]_{\psi(i),\psi(i)}dt\bigg)^{q/2}\bigg] \\
&\quad \quad +C_q\sum_{l=1}^2E_\Pi\bigg[\bigg(\sum_{\substack{i_1,i_2,j_1,j_2 \\ \psi(i_1)=1, \psi(i_2)=2}}[S_n^{-1}(\sigma)]_{i_1,j_1}[S_n^{-1}(\sigma)]_{i_2,j_2}
\int_{I_{i_1}\cap I_{i_2}} \Delta_{j_1,t} X^c \Delta_{j_2,t} X^c[\Sigma_t]_{\psi(i_1),\psi(i_2)}dt\bigg)^{q/2}\bigg] \\
&\quad \leq C_qE_\Pi\bigg[\bigg(\sum_i\frac{\sup_t|\Delta_{i,t} X^c|^2}{|I_i|}\bigg\lVert \mcd^{1/2}S_n^{-1}(\sigma)\mcd^{1/2}
\MAT{cc}{\mce & G \\ G^\top & \mce} \mcd^{1/2}S_n^{-1}(\sigma)\mcd^{1/2}\bigg\rVert\bigg)^{q/2}\bigg] \\
&\quad \leq C_qM_n^{q/2}(1-\bar{\rho}_n)^q
}
on $\{\bar{\rho}_n<1\}$ for $q\geq 1$.

Then, thanks to (\ref{XSX-martingale-est}), we obtain
\EQ{\label{XSX-martingale-est2} \Delta X^c\INVS\Delta X^c - {\rm tr}(\INVS S_n(\sigma_0))=\bar{R}_n(\sqrt{n}).}

(\ref{XSX-martingale-est2}), (\ref{H1n-eq}) and similar estimates for $\PS^k(H_n^1(\sigma)-H_n^1(\sigma_0))$ yield
\EQNN{ &\PS^k(H_n^1(\sigma)-H_n^1(\sigma_0)) \\
&\quad =-\frac{1}{2}\PS^k{\rm tr}(S_n(\sigma_0)(\INVS-\INVSz))-\frac{1}{2}\PS^k\log\frac{\det S_n(\sigma)}{\det S_n(\sigma_0)}+\bar{R}_n(\sqrt{n}) \\
&\quad =-\frac{1}{2}\PS^k{\rm tr}(\INVS(S_n(\sigma_0)-S_n(\sigma)))-\frac{1}{2}\PS^k\log\frac{\det S_n(\sigma)}{\det S_n(\sigma_0)}+\bar{R}_n(\sqrt{n})
}
for $k\in \{0,1,2,3,4\}$.
Therefore, Sobolev's inequality yields the conclusion.

\end{proof}

Let $\mcy_1(\sigma)=\lim_{T\to\infty}(T^{-1}\int_0^T y_{1,t}(\sigma)dt)$, where
\EQNN{y_{1,t}(\sigma)=-\frac{1}{2}\mca(\rho_t)\sum_{l=1}^2B_{l,t}^2+\mca(\rho_t)\frac{B_{1,t}B_{2,t}\rho_{t,0}}{\rho_t}
+\SUMLL a_0^l\bigg(\frac{1}{2}-\frac{1}{2}B_{l,t}^2+\log B_{l,t}\bigg)+\int_{\rho_{t,0}}^{\rho_t}\frac{\mca(\rho)}{\rho}d\rho.
}
The limit $\mcy_1(\sigma)$ exists under (A1), (A3) and (A4).

\begin{proposition}\label{H1n-conv-prop}
Assume (A1)--(A4). Then
\EQQ{\sup_{\sigma\in\Theta_1} |n^{-1}\PS^k(H_n^1(\sigma)-H_n^1(\sigma_0))-\PS^k\mcy_1(\sigma)|\TOP 0}
as $n\to\infty$ for $k\in \{0,1,2,3\}$.
\end{proposition}

\begin{proof}
Let $\mca_p^1=(\TILG\TILG^\top)^p$, $\mca_p^2=(\TILG^\top\TILG)^p$, $\TSIG_{i,0}^l=\TSIG_i^l(\sigma_0)$ and $\TSIG_{i,j,0}^{1,2}=\TSIG_{i,j}^{1,2}(\sigma_0)$.
Thanks to (A1), for any $\epsilon>0$, there exists $\delta>0$ such that $|t-s|<\delta$ implies
\EQ{\label{uniform-conti-eq} |\rho_t-\rho_s|\vee |\Sigma_t-\Sigma_s|\vee |\mu_t-\mu_s|<\epsilon}
for any $\sigma$ and $\theta$. We fix such $\delta>0$, and fix a partition $s_k= k\delta/2$.
Then, (\ref{InvS-eq}) and (A4) yield
\EQN{\label{H1n-est-eq1} &n^{-1}{\rm tr}(\INVS(S_n(\sigma_0)-S_n(\sigma))) \\
&\quad =\frac{1}{n}{\rm tr}\bigg(\INVS \TMCD^{1/2}
\MAT{cc}{{\rm diag}((\TSIG_{i,0}^1-\TSIG_i^1)_i) & \{(\TSIG_{i,j,0}^{1,2}-\TSIG_{i,j}^{1,2})[G]_{ij}\}_{ij} \\ 
\{(\TSIG_{i,j,0}^{1,2}-\TSIG_{i,j}^{1,2})[G]_{ij}\}_{ji} & {\rm diag}((\TSIG_{j,0}^2-\TSIG_j^2)_j)}\TMCD^{1/2}\bigg) \\
&\quad =\frac{1}{n}\sum_{p=0}^\infty \bigg\{\sum_{l=1}^2{\rm tr}\bigg({\rm diag}\bigg(\bigg(\frac{\TSIG_{i,0}^l}{\TSIG_i^l}-1\bigg)_i\bigg)\mca_p^l\bigg)
-2{\rm tr}\bigg(\mca_p^1\TILG \bigg\{\frac{\TSIG_{i,j,0}^{1,2}-\TSIG_{i,j}^{1,2}}{(\TSIG_i^1)^{1/2}(\TSIG_j^2)^{1/2}}[G^\top]_{ij}\bigg\}_{ij}\bigg) \bigg\} \\
&\quad =\frac{1}{n}\sum_{p=0}^\infty\sum_{k=1}^{q_n} \bigg\{\sum_{l=1}^2{\rm tr}\bigg({\rm diag}\bigg(\bigg(\frac{\TSIG_{i,0}^l}{\TSIG_i^l}-1\bigg)_i\bigg)\EK^1\mca_p^l\bigg)
-2{\rm tr}\bigg(\EK^1\mca_p^1\TILG \bigg\{\frac{\TSIG_{i,j,0}^{1,2}-\TSIG_{i,j}^{1,2}}{(\TSIG_i^1)^{1/2}(\TSIG_j^2)^{1/2}}[G^\top]_{ij}\bigg\}_{ij}\bigg) \bigg\}.
}

Let $\dot{\rho}_k=\rho_{s_{k-1}}$, $\dot{B}_{k,l}=([\Sigma_{s_{k-1}}(\sigma_0)]_{ll}/[\Sigma_{s_{k-1}}(\sigma)]_{ll})^{1/2}$, 
$\dot{\mca}_{k,p}^1=\EK^1(GG^\top)^p$ and $\dot{\mca}_{k,p}^2=\EK^2(G^\top G)^p$.
Then, (\ref{uniform-conti-eq}) yields that for any $p\in\mbbz_+$, we have
\EQ{\label{mca-est} |[\EK^l\mca_p^l]_{ij}-\dot{\rho}_k^{2p}[\dot{\mca}_{k,p}^l]_{ij}|\leq Cp\bar{\rho}_n^{2p-1}\epsilon}
if $2pr_n<\delta/2$.
Moreover, Lemmas~\ref{G-est-lemma} and~\ref{TILG-est-lemma} and (\ref{bar-p-prob}) yield
\EQ{\label{GGT-sum-est} \limsup_{n\to\infty} \max_{1\leq k\leq q_n+1} \sum_{p=0}^\infty \lVert\EK^l\mca_p^l\rVert\leq C\limsup_{n\to\infty}\sum_{p=0}^\infty\bar{\rho}_n^{2p}<\infty.}
Then, together with (A2), we obtain
\EQN{\label{H1n-conv-eq1} &n^{-1}{\rm tr}(\INVS(S_n(\sigma_0)-S_n(\sigma))) \\
&\quad =\frac{1}{n}\sum_{p=0}^\infty \SUMK\bigg\{\dot{\rho}_k^{2p}\sum_{l=1}^2(\dot{B}_{k,l}^2-1){\rm tr}(\dot{\mca}_{k,p}^l)
-2\dot{\rho}_k^{2p+1}(\dot{B}_{k,1}\dot{B}_{k,2}\dot{\rho}_{k,0}-\dot{\rho}_k){\rm tr}(\dot{\mca}_{k,p+1}^1)\bigg\}+e_n,
}
where $\dot{\rho}_{k,0}=\rho_{s_{k-1}}(\sigma_0)$, and $(e_n)_{n=1}^\infty$ denotes a general sequence of random variables such that 
$\limsup_{n\to\infty}|e_n|\to 0$ as $\delta \to 0$.

Moreover, (\ref{bar-p-prob}), Lemma~\ref{TILG-est-lemma}, Lemma A.3 in~\cite{ogi18} yield
\EQNN{\log \det S_n(\sigma) 
&= \log \det \TMCD + \log \det \bigg(\mce_M+\MAT{cc}{0 & \TILG \\ \TILG^\top & 0}\bigg) \\
&= \SUMLL\sum_{i=1}^{M_l} \log\TSIG_i^l+ \sum_{p=1}^\infty \frac{(-1)^{p-1}}{p}{\rm tr}\bigg(\MAT{cc}{0 & \TILG \\ \TILG^\top & 0}^p\bigg) \\
&= \SUMLL\sum_{i=1}^{M_l} \log\TSIG_i^l - \sum_{p=1}^\infty \frac{1}{p}{\rm tr}((\TILG\TILG^\top)^p).
} 
Therefore, thanks to (\ref{mca-est}), we obtain
\EQN{\label{H1n-conv-eq2} n^{-1}\log \frac{\det S_n(\sigma)}{\det S_n(\sigma_0)}
&=n^{-1}\SUMLL\sum_{i=1}^{M_l}\log\frac{\TSIG_i^l}{\TSIG_{i,0}^l}-n^{-1}\sum_{p=1}^\infty \frac{1}{p}{\rm tr}((\TILG\TILG^\top)^p-(\TILG_0\TILG_0^\top)^p) \\
&=-n^{-1}\SUMK\bigg\{\SUMLL M_{l,k}\log\dot{B}_{k,l}^2+\sum_{p=1}^\infty \frac{\dot{\rho}_k^{2p}-\dot{\rho}_{k,0}^{2p}}{p}{\rm tr}(\dot{\mca}_{k,p}^1)\bigg\}+e_n.
}

(\ref{H1n-eq2}), (\ref{H1n-conv-eq1}) and (\ref{H1n-conv-eq2}) yield
\EQN{\label{H1-est-eq} H_n^1(\sigma)-H_n^1(\sigma_0)&=\SUMK\bigg\{-\frac{1}{2}\SUMP\dot{\rho}_k^{2p}\SUMLL\dot{B}_{k,l}^2{\rm tr}(\dot{\mca}_{k,p}^l)
+\SUMPP \dot{\rho}_k^{2p-1}\dot{\rho}_{k,0}\dot{B}_{k,1}\dot{B}_{k,2}{\rm tr}(\dot{\mca}_{k,p}^1)+\frac{1}{2}\SUMLL{\rm tr}(\dot{A}_{k,0}^l) \\
&\qquad \qquad +\SUMLL M_{l,k}\log\dot{B}_{k,l}+\sum_{p=1}^\infty \frac{\dot{\rho}_k^{2p}-\dot{\rho}_{k,0}^{2p}}{2p}{\rm tr}(\dot{\mca}_{k,p}^1)
\bigg\}+ne_n \\
&=n\mcy_1+ne_n.
}

Together with (A3) and a similar estimates for $\PS^k(H_n^1(\sigma)-H_n^1(\sigma_0))$, we have
\EQQ{n^{-1}\PS^k(H_n^1(\sigma)-H_n^1(\sigma_0))\TOP \PS^k\mcy_1(\sigma)}
for $k\in\{0,1,2,3,4\}$.
\begin{discuss}
{\colorr (\ref{H1n-est-eq1})の両辺は等式だから微分できる．(\ref{mca-est}), (\ref{GGT-sum-est})を使えば(\ref{H1n-conv-eq1})も同様なのでOK.}
\end{discuss}
Then, Sobolev's inequality yields the conclusion.

\end{proof}

\begin{proposition}\label{mcy1-est-prop}
There exists a positive constant $\chi$ such that
\EQQ{\mcy_1\leq \liminf_{T\to\infty}\int_0^T\bigg\{-\frac{1}{2}(a_0^1\wedge a_0^2)(B_{1,t}-B_{2,t})^2-\chi \big\{a_1^1(\rho_t-\rho_{t,0})^2+a_0^1\wedge a_0^2 (B_{1,t}B_{2,t}-1)^2\big\}\bigg\}dt.}
\end{proposition}

\begin{proof}
The proof is based on the ideas of proof of Lemma 5 in~\cite{ogi-yos14}.
Let
\EQQ{G_k=\{[G]_{ij}1_{\{\sup I_i^1, \sup I_j^2 \in (s_{k-1}, s_k]\}}\}_{ij},}
and let $\tilde{\mca}_{k,p}^l$ be obtained similarly to $\dot{\mca}_{k,p}^l$ replacing $\mce_{(k)}(GG^\top)^p$ by $(G_kG_k^\top)^p$.
Let $\tilde{\mca}_k=\SUMPP \dot{\rho}_k^{2p}\tilde{\mca}_{k,p}^1$ and $\tilde{\mcb}_k=\sum_{p=1}^\infty (2p)^{-1}(\dot{\rho}_k^{2p}-\dot{\rho}_{k,0}^{2p}){\rm tr}(\tilde{\mca}_{k,p}^1)$, then we have
\EQNN{\mcy_1&=n^{-1}\SUMK \bigg\{-\frac{1}{2}(M_{1,k}+\tilde{\mca}_k)(\dot{B}_{k,1}-\dot{B}_{k,2})^2+M_{1,k}(1+\log (\dot{B}_{k,1}\dot{B}_{k,2})) \\
&\quad +\tilde{\mcb}_k+\frac{M_{2,k}-M_{1,k}}{2}(1-\dot{B}_{k,2}^2+\log(\dot{B}_{k,2}^2)) 
+\dot{B}_{k,1}\dot{B}_{k,2}\bigg(\tilde{\mca}_k\frac{\dot{\rho}_{k,0}}{\dot{\rho}_k}-\tilde{\mca}_k-M_{1,k}\bigg)\bigg\}+e_n \\
&=n^{-1}\SUMK\bigg\{-\frac{1}{2}(M_{2,k}+\dot{\mca}_k)(\dot{B}_{k,1}-\dot{B}_{k,2})^2+M_{2,k}(1+\log (\dot{B}_{k,1}\dot{B}_{k,2})) \\
&\quad +\tilde{\mcb}_k+\frac{M_{1,k}-M_{2,k}}{2}(1-\dot{B}_{k,1}^2+\log(\dot{B}_{k,1}^2)) 
+\dot{B}_{k,1}\dot{B}_{k,2}\bigg(\tilde{\mca}_k\frac{\dot{\rho}_{k,0}}{\dot{\rho}_k}-\tilde{\mca}_k-M_{2,k}\bigg)\bigg\}+e_n.
}

For $l\in\{1,2\}$, let
\EQQ{F_{l,k}=M_{l,k}(1+\log(\dot{B}_{k,1}\dot{B}_{k,2}))+\tilde{\mcb}_k+\dot{B}_{k,1}\dot{B}_{k,2}\bigg(\tilde{\mca}_k\frac{\dot{\rho}_{k,0}}{\dot{\rho}_k}-\tilde{\mca}_k-M_{l,k}\bigg),}
then we obtain
\EQ{\label{mcy1n-est} \mcy_1\leq n^{-1}\SUMK\bigg\{-\frac{1}{2}(M_{1,k}\wedge M_{2,k}+\tilde{\mca}_k)(\dot{B}_{k,1}-\dot{B}_{k,2})^2+F_{1,k}\vee F_{2,k}\bigg\}+e_n.}
\begin{discuss}
{\colorr OgiYos14のLemma 3のように(\ref{H1n-conv-eq1})と$\mce_{(k)}(GG^\top)^p$を$(G_kG_k^\top)^p$で置き換えたものの差が$s_k$から$2pr_n$はなれた部分の$M_{l,k}$で上から抑えられるから収束．}
\end{discuss}

Let $(\lambda_i^k)_{i=1}^{M_{1,k}}$ be all the eigenvalues of $G_kG_k^\top$. Then, we have
\EQNN{F_{1,k}&=\sum_{i=1}^{M_{1,k}}\bigg\{1+\log (\dot{B}_{k,1}\dot{B}_{k,2})+\dot{B}_{k,1}\dot{B}_{k,2}\SUMP \big\{(\lambda_i^k)^{p+1}\dot{\rho}_k^{2p+1}\dot{\rho}_{k,0}-(\lambda_i^k)^p\dot{\rho}_k^{2p}\big\}
+\SUMPP \frac{(\lambda_i^k)^p}{2p}(\dot{\rho}_k^{2p}-\dot{\rho}_{k,0}^{2p})\bigg\}.
}

Moreover, by setting
$g_i^k=\sqrt{1-\lambda_i^k\dot{\rho}_k^2}$, $g_{i,0}^k=\sqrt{1-\lambda_i^k\dot{\rho}_{k,0}^2}$, and $F(x)=1-x+\log x$, we have
\EQNN{F_{1,k}&=\sum_{i=1}^{M_{1,k}}\Big\{1+\dot{B}_{k,1}\dot{B}_{k,2}(g_i^k)^{-2}(\lambda_i^k\dot{\rho}_k\dot{\rho}_{k,0}-1)+\log (\dot{B}_{k,1}\dot{B}_{k,2}g_{i,0}^k(g_i^k)^{-1})\Big\} \\
&=\sum_{i=1}^{M_{1,k}}\Big\{\dot{B}_{k,1}\dot{B}_{k,2}(g_i^k)^{-2}(\lambda_i^k\dot{\rho}_k\dot{\rho}_{k,0}-1)+\dot{B}_{k,1}\dot{B}_{k,2}g_{i,0}^k(g_i^k)^{-1}+F(\dot{B}_{k,1}\dot{B}_{k,2}g_{i,0}^k(g_i^k)^{-1})\Big\}.
}

Let 
\EQQ{\mcr=\sup_{t,\sigma,k}(|\PS^k \Sigma|\vee|\PS^k \Sigma^{-1}|).}
Since $g_i^k\leq 1$, $0\leq \lambda_i^k\leq 1$, and $|\dot{\rho}_k|\leq 1$, we have
\EQNN{(g_i^k)^{-2}(\lambda_i^k\dot{\rho}_k\dot{\rho}_{k,0}-1)
&= -\frac{(\lambda_i^k\dot{\rho}_k\dot{\rho}_{k,0}-1)^2-(g_{i,0}^k)^2(g_i^k)^2}{(g_i^k)^2(1-\lambda_i^k\dot{\rho}_k\dot{\rho}_{k,0}+g_{i,0}^kg_i^k)} \\
&= -\frac{\lambda_i^k(\dot{\rho}_k-\dot{\rho}_{k,0})^2}{(g_i^k)^2(1-\lambda_i^k\dot{\rho}_k\dot{\rho}_{k,0}+g_{i,0}^kg_i^k)} \\
&\leq -\frac{\lambda_i^k}{3}(\dot{\rho}_k-\dot{\rho}_{k,0})^2.
}
Together with Lemma 11 in~\cite{ogi-yos14} and 
\EQQ{\dot{B}_{k,1}\dot{B}_{k,2}g_{i,0}^k(g_i^k)^{-1}-1\leq \frac{\mcr^4}{\sqrt{1-\bar{\rho}_n^2}},}
we have
\EQQ{F_{1,k}\leq \sum_{i=1}^{M_{1,k}}\bigg\{-\frac{\dot{B}_{k,1}\dot{B}_{k,2}}{3}\lambda_i^k(\dot{\rho}_k-\dot{\rho}_{k,0})^2-\frac{1-\bar{\rho}_n^2}{4\mcr^8}(\dot{B}_{k,1}\dot{B}_{k,2}g_{i,0}^k(g_i^k)^{-1}-1)^2\bigg\}.
}

Moreover, since 
\EQNN{(\dot{B}_{k,1}\dot{B}_{k,2}g_{i,0}^k(g_i^k)^{-1}-1)^2
&\geq (\dot{B}_{k,1}\dot{B}_{k,2}g_{i,0}^k-g_i^k)^2 \\
&\geq \frac{(g_{i,0}^k)^2}{2}(\dot{B}_{k,1}\dot{B}_{k,2}-1)^2 - (g_i^k-g_{i,0}^k)^2 \\
&=\frac{1-\bar{\rho}_n^2}{2}(\dot{B}_{k,1}\dot{B}_{k,2}-1)^2-\frac{(\lambda_i^k)^2(\dot{\rho}_k-\dot{\rho}_{k,0})^2(\dot{\rho}_k+\dot{\rho}_{k,0})^2}{(g_i^k+g_{i,0}^k)^2} \\
&\geq \frac{1-\bar{\rho}_n^2}{2}(\dot{B}_{k,1}\dot{B}_{k,2}-1)^2 - \frac{\lambda_i^k}{1-\bar{\rho}_n^2}(\dot{\rho}_k-\dot{\rho}_{k,0})^2,
}
we have
\EQNN{F_{1,k}&\leq \sum_{i=1}^{M_{1,k}}\bigg\{-\frac{\dot{B}_{k,1}\dot{B}_{k,2}}{3}\lambda_i^k(\dot{\rho}_k-\dot{\rho}_{k,0})^2 -\frac{(1-\bar{\rho}_n^2)^2}{8\mcr^8}(\dot{B}_{k,1}\dot{B}_{k,2}-1)^2 +\frac{\lambda_i^k}{4\mcr^8}(\dot{\rho}_k-\dot{\rho}_{k,0})^2\bigg\} \\
&=-\bigg(\frac{\dot{B}_{k,1}\dot{B}_{k,2}}{3}-\frac{1}{4\mcr^8}\bigg)\dot{\mca}_{k,1}^1(\dot{\rho}_k-\dot{\rho}_{k,0})^2
-\frac{(1-\bar{\rho}_n^2)^2}{8\mcr^8}M_{1,k}(\dot{B}_{k,1}\dot{B}_{k,2}-1)^2.
}
By a similar argument for $F_{2,k}$, there exists a positive random variable $\chi$ which does not depend on $k$ nor $n$ such that
\EQQ{F_{1,k}\vee F_{2,k} \leq -\chi \big\{\dot{\mca}_{k,1}^1(\dot{\rho}_k-\dot{\rho}_{k,0})^2 + M_{1,k}\wedge M_{2,k} (\dot{B}_{k,1}\dot{B}_{k,2}-1)^2\big\}.}

Together with (\ref{mcy1n-est}), we have
\EQNN{\mcy_{1,n}\leq n^{-1}\SUMK\bigg\{-\frac{1}{2}(M_{1,k}\wedge M_{2,k})(\dot{B}_{k,1}-\dot{B}_{k,2})^2 - \chi \big\{\dot{\mca}_{k,1}^1(\dot{\rho}_k-\dot{\rho}_{k,0})^2 + M_{1,k}\wedge M_{2,k} (\dot{B}_{k,1}\dot{B}_{k,2}-1)^2\big\}\bigg\}.
}
By letting $n\to \infty$, (A4) and (A6) yield the conclusion.

\begin{discuss}
{\colorr 
\EQQ{\bigg|n^{-1}\sum_kM_{1,k}\dot{B}_{k,1}-q_n\sum_ka_0^1\dot{B}_{k,1}\bigg|
\leq \bigg|\sum_k(n^{-1}M_{1,k}-q_n^{-1}a_0^1)\dot{B}_{k,1}\bigg|
\leq \max_k|q_nn^{-1}M_{1,k}-a_0^1|\cdot q_n^{-1}\sum_k\dot{B}_{k,1}}
などからOK.
}
\end{discuss}
\end{proof}

(A6) and Remark 4 in~\cite{ogi-yos14} yield that
\EQQ{\limsup_{T\to\infty}\frac{1}{T}\int_0^T\big\{|B_{1,t}-B_{2,t}|^2+|B_{1,t}B_{2,t}-1|^2+\lVert \rho_t-\rho_{t,0}\rVert^2\big\} dt>0,}
when $\sigma \neq \sigma_0$.

Then, by Proposition~\ref{mcy1-est-prop}, we have $\mcy_1(\sigma)<0$.
Therefore, for any $\epsilon,\delta>0$, there exists $\eta>0$ such that
\EQQ{P\bigg(\inf_{|\sigma-\sigma_0|\geq \delta}(-\mcy_1(\sigma))<\eta\bigg)<\frac{\epsilon}{2}.}

Then, since $H_n^1(\HSN)-H_n^1(\sigma_0)\geq 0$ by the definition, we have
\EQ{\label{sigma-consis-eq} P(|\HSN-\sigma_0|\geq \delta)\leq P\bigg(\inf_{|\sigma-\sigma_0|\geq \delta} (-\mcy_1(\sigma))<\eta\bigg)
+P\bigg(\sup_\sigma |n^{-1}(H_n^1(\sigma)-H_n^1(\sigma_0))-\mcf_1(\sigma)|\geq \eta\bigg)<\eta
}
by Proposition~\ref{H1n-conv-prop}, which implies $\hat{\sigma}_n\TOP \sigma_0$ as $n\to\infty$.

\subsection{Asymptotic normality of $\hat{\sigma}_n$}\label{HSN-clt-subsection}

Let $S_{n,0}=S_n(\sigma_0)$ and $\Sigma_{t,0}=\Sigma_t(\sigma_0)$.
(\ref{H1n-eq}) implies
\EQN{\label{H1n-eq3} \PS H_n^1(\sigma_0)&=-\frac{1}{2}(\Delta X^c)^\top\PSINVS \Delta X^c-\frac{1}{2}{\rm tr}(\PS S_{n,0} S_{n,0}^{-1})+o_p(\sqrt{n}) \\
&=-\frac{1}{2}{\rm tr}(\PSINVS(\Delta X^c(\Delta X^c)^\top - S_{n,0}))+o_p(\sqrt{n}).}

Let $(L_n)_{n\in \mbbn}$ be a sequence of positive integers such that $L_n\to \infty$ and $L_n(nh_n)^{-1}\to 0$ as $n\to\infty$.
\begin{discuss}
{\colorr この条件は$\theta$のCLTで必要} 
\end{discuss}
Let $\check{s}_k=kT_n/L_n$ for $0\leq k\leq L_n$, let $J^k=(\check{s}_{k-1},\check{s}_k]$, and let $\SL$ be an $M\times M$ matrix satisfying 
\EQQ{[\SL]_{ij}=\int_{I_i\cap I_j\cap J_k}[\Sigma_{t,0}]_{ij}dt.}

For a two-dimensional stochastic process $(U_t)_{t\geq 0}=((U_t^1,U_t^2))_{t\geq 0}$,
let $\Delta_{i,t}^{l,(k)} U=U^l_{(S^{n,l}_i\vee \check{s}_{k-1})\wedge \check{s}_k\wedge t}-U^l_{(S^{n,k}_{i-1}\vee \check{s}_{k-1})\wedge \check{s}_k\wedge t}$, 
and let $\Delta_{i,t}^{(k)} U=\Delta_{\varphi(i),t}^{\psi(i),(k)} U$ for $1\leq i\leq M$. 
Let $\Delta_i^{(k)} U=\Delta_{i,T_n}^{(k)} U$, and let $\Delta^{(k)} U=(\Delta_i^{(k)} U)_{1\leq i\leq M}$.

Let 
\EQQ{\mcx_k=-\frac{1}{2\sqrt{n}}\big\{(\Delta^{(k)} X^c)^\top\PSINVS \Delta^{(k)}X^c-{\rm tr}(\PSINVS S_0^{(k)})\big\} -\frac{1}{\sqrt{n}}\sum_{k'<k}(\Delta^{(k)} X^c)^\top\PSINVS \Delta^{(k')}X^c.}

Then since $\Delta X^c=\SUML\Delta^{(k)} X^c$ and $S_{n,0}=\SUML\SL$, (\ref{H1n-eq3}) yields
\EQ{\label{H1n-eq4} n^{-1/2}\PS H_n^1(\sigma_0)=\SUML\mcx_k+o_p(1).}

Moreover, It\^o's formula yields
\EQN{\label{mcx-eq} \sqrt{n}\mcx_k&=-\frac{1}{2}\sum_{i,j}[\PSINVS]_{ij}\bigg\{2\int_{I_i\cap J^k}\Delta_{j,t}^{(k)} X^cdX_t^{c,\psi(i)}
+2\sum_{k'<k}\int_{I_i\cap J^k}\Delta_j^{(k')} X^cdX_t^{c,\psi(i)}\bigg\} \\
&=-\sum_{i,j}[\PSINVS]_{ij}\int_{I_i\cap J^k}\Delta_{j,t} X^cdX_t^{c,\psi(i)}.
}

Let $\mcg_t=\mcf_t\bigvee\sigma(\{\Pi_n\}_n)$ for $t\geq 0$.
We will show 
\EQ{\label{H1n-conv} n^{-1/2}\PS H_n^1(\sigma_0) \TOD N(0,\Gamma_1),}
by using Corollary 3.1 and the remark after that in Hall and Heyde~\cite{hal-hey80}.
For this purpose, it is sufficient to show 
\EQ{\label{mcx2-conv} \SUML E_k[\mcx_k^2]\TOP \Gamma_1,}
and
\EQ{\label{mcx4-conv} \SUML E_k[\mcx_k^4]\TOP 0,}
by (\ref{H1n-eq4}), where $E_k$ denotes the conditional expectation with respect to $\mcg_{\check{s}_{k-1}}$.

We first show some auxiliary lemmas. Let $\tilde{M}_k=\#\{i; 1\leq i\leq M, \sup I_i\in J_k\}$.
\begin{lemma}\label{Sl-est-lemma}
Assume (A1). Then, there exists a positive constant $C$ such that
$\lVert \mcd^{-1/2}\SL \mcd^{-1/2}\rVert \leq C$ and ${\rm tr}(\mcd^{-1/2}\SL \mcd^{-1/2})\leq C(\tilde{M}_k+1)$ for any $1\leq k\leq L_n$.
\end{lemma}
\begin{proof}
Since
\EQQ{[\SL]_{ij}\leq C\bigg[\mcd^{1/2}\MAT{cc}{\mce_{M_1} & G \\ G^\top & \mce_{M_2}}\mcd^{1/2}\bigg]_{ij},}
Lemma~\ref{G-est-lemma} yields
\EQQ{\lVert \mcd^{-1/2}\SL \mcd^{-1/2}\rVert \leq C\bigg \lVert \MAT{cc}{\mce_{M_1} & G \\ G^\top & \mce_{M_2}}\bigg \rVert \leq C.}
Moreover, we have
\EQQ{{\rm tr}(\mcd^{-1/2}\SL \mcd^{-1/2})=\sum_{i=1}^M\frac{\int_{I_i\cap J^k}[\Sigma_{t,0}]_{\psi(i),\psi(i)}dt}{|I_i|}\leq C \sum_{i=1}^M1_{\{i; I_i\cap J^k\neq \emptyset\}}
\leq C(\tilde{M}_k+1).}
\end{proof}
	
\begin{lemma}\label{TILM-tight-lemma}
Assume (A4) and that $nh_nL_n^{-1}\to \infty$ as $n\to\infty$. Then, $\{L_nn^{-1}\max_{1\leq k\leq L_n}\tilde{M}_k\}_{n=1}^\infty$ is $P$-tight.
\end{lemma}

\begin{proof}
Let $\mcm_n=[nh_nL_n^{-1}]$. We define a partition of $[0,\infty)$ by 
\EQQ{s_j=\frac{nh_nj}{2L_n\mcm_n} \quad (0\leq j\leq 2L_n\mcm_n).}
Then, $(s_j)_{j=0}^\infty \in \mfs$ when $nh_nL_n^{-1}\geq 1$.

For $M_{l,j}$ which corresponds to this partition, we have
\EQQ{\tilde{M}_k\leq \sum_{l=1}^2\sum_{j=2\mcm_n(k-1)+1}^{2\mcm_nk} M_{l,j},}
since $nh_nkL_n^{-1}=s_{2\mcm_nk}$.
Therefore, we obtain
\EQQ{\max_{1\leq k\leq L_n}\tilde{M}_k\leq 4\mcm_n \max_{l,j}M_{l,j}
\leq 4\mcm_n\{h_n^{-1}(a_0^1\vee a_0^2)+o_p(h_n^{-1})\}=O_p(nL_n^{-1}).}
\end{proof}

\begin{lemma}\label{Sl-est-lemma2}
Assume (A1). Then, 
\EQQ{\lVert \tilde{\mcd}^{-1/2}\SL \PSINVS S_{n,0}^{(k')}\tilde{\mcd}^{-1/2}\rVert \leq C\frac{\mcq_n\bar{\rho}_n^{\mcq_n}}{(1-\bar{\rho}_n)^2}}
on $\{\bar{\rho}_n<1\}$ for $|k-k'|>1$, where $\mcq_n=[r_n^{-1}(T_n/L_n-2r_n)]$.
\end{lemma}

\begin{proof}
By using the expansion formula (\ref{InvS-eq}), we have
\EQN{\label{SlSl'-est} \SL \PSINVS S_{n,0}^{(k')}&=-\SL S_{n,0}^{-1} \PS S_{n,0}S_{n,0}^{-1} S_{n,0}^{(k')} \\
&=-\SL \RITD\sum_{p=0}^\infty (-1)^p\MAT{cc}{0 & \TILG \\ \TILG^\top & 0}^p \RITD\PS S_{n,0}\RITD\sum_{q=0}^\infty (-1)^q\MAT{cc}{0 & \TILG \\ \TILG^\top & 0}^q \RITD S_{n,0}^{(k')} \\
&=-\sum_{p,q=0}^\infty (-1)^{p+q+1}\SL \RITD\MAT{cc}{0 & \TILG \\ \TILG^\top & 0}^p \RITD\PS S_{n,0}\RITD\MAT{cc}{0 & \TILG \\ \TILG^\top & 0}^q \RITD S_{n,0}^{(k')}.
}

The element
\EQ{\label{Sl-est-lemma-eq1} \bigg[\RITD\MAT{cc}{0 & \TILG \\ \TILG^\top & 0}^p \RITD\PS S_{n,0}\RITD\MAT{cc}{0 & \TILG \\ \TILG^\top & 0}^q\RITD\bigg]_{ij}}
is equal to zero if $[\bar{S}^{p+q+1}]_{ij}=0$. Moreover, 
$[\SL]_{i'i}\neq 0$ only if $I_i\cap J^k \neq \emptyset$, and $[S_{n,0}^{(k')}]_{jj'}\neq 0$ only if $I_j\cap J^{k'}\neq \emptyset$.
Since $\inf_{x\in I_i, y\in I_j}|x-y|> T_n/L_n-2r_n$ if $I_i\cap J^k \neq \emptyset$ and $I_j\cap J^{k'}\neq \emptyset$,
we have $[\bar{S}^r]_{ij}=0$ for $r\leq \mcq_n$ in this case.

Therefore, all the elements (\ref{Sl-est-lemma-eq1}) are zero if $p+q+1\leq \mcq_n$. Then, (\ref{SlSl'-est}) and Lemmas~\ref{TILG-est-lemma},~\ref{Sn-lemma} and \ref{Sl-est-lemma} yield
\EQNN{&\lVert \tilde{\mcd}^{-1/2}\SL \PSINVS S_{n,0}^{k'}\tilde{\mcd}^{-1/2}\rVert \\
&\quad \leq \sum_{p=0}^\infty\sum_{q=(\mcq_n-p)\vee 0}^\infty \bigg\lVert \tilde{\mcd}^{-1/2}\SL \RITD\MAT{cc}{0 & \TILG \\ \TILG^\top & 0}^p \RITD\PS S_{n,0}\RITD\MAT{cc}{0 & \TILG \\ \TILG^\top & 0}^q \RITD S_{n,0}^{(k')}\tilde{\mcd}^{-1/2}\bigg\rVert \\
&\quad \leq C\sum_{p=0}^\infty\sum_{q=(\mcq_n-p)\vee 0}^\infty \bar{\rho}_n^{p+q}=C\frac{\mcq_n\bar{\rho}_n^{\mcq_n}+\bar{\rho}_n^{\mcq_n}(1-\bar{\rho}_n)^{-1}}{1-\bar{\rho}_n}  \\
&\quad \leq C\frac{\mcq_n\bar{\rho}_n^{\mcq_n}}{(1-\bar{\rho}_n)^2}
}
on $\{\bar{\rho}_n<1\}$.
\end{proof}

\begin{proposition}\label{Hn1-st-conv-prop}
Assume (A1)--(A4) and (A6). Then,
\EQQ{n^{-1/2}\PS H_n^1(\sigma_0) \TOD N(0,\Gamma_1),}
as $n\to\infty$.
\end{proposition}

\begin{proof}
It is sufficient to show (\ref{mcx2-conv}) and (\ref{mcx4-conv}).
Let $\mfa_k=(\Delta^{(k)} X^c)^\top\PSINVS \Delta^{(k)}X^c$ and $\mfb_k=\PSINVS \SL$. By the definition of $\mcx_k$, we have
\EQNN{&\SUML E_k[\mcx_k^4] \\
&\quad \leq \frac{C}{n^2}\SUML \bigg\{E_k\big[\big\{(\Delta^{(k)} X^c)^\top\PSINVS \Delta^{(k)}X^c-{\rm tr}(\PSINVS \SL)\big\}^4\big]
+E_k\bigg[\bigg(\sum_{k'<k}(\Delta^{(k)} X^c)^\top\PSINVS \Delta^{(k')}X^c\bigg)^4\bigg]\bigg\} \\
&\quad =\frac{C}{n^2}\SUML \Big\{E_k[\mfa_k^4]-4E_k[\mfa_k^3]{\rm tr}(\mfb_k)+6E_k[\mfa_k^2]{\rm tr}(\mfb_k)^2-4{\rm tr}(\mfb_k)^4+{\rm tr}(\mfb_k)^4\Big\} \\
&\quad \quad +\frac{C}{n^2}\SUML \bigg\{\bigg(\sum_{k'<k}\Delta^{(k')}X^c\bigg)^\top \PSINVS \SL \PSINVS \bigg(\sum_{k'<k}\Delta^{(k')}X^c\bigg)\bigg\}^2.
}

Thanks to Lemmas~\ref{XAX-lemma},~\ref{TILM-tight-lemma} and \ref{Sn-lemma}, the first term in the right-hand side is calculated as 
\EQNN{&\frac{C}{n^2}\SUML \Big\{{\rm tr}(\mfb_k)^4+12{\rm tr}(\mfb_k)^2{\rm tr}(\mfb_k^2)+12{\rm tr}(\mfb_k^2)^2+32{\rm tr}(\mfb_k){\rm tr}(\mfb_k^3)+48{\rm tr}(\mfb_k^4) \\
&\qquad \qquad -4{\rm tr}(\mfb_k)\big\{{\rm tr}(\mfb_k)^3+6{\rm tr}(\mfb_k){\rm tr}(\mfb_k^2)+8{\rm tr}(\mfb_k^3)\big\}
+6{\rm tr}(\mfb_k)^2\big\{{\rm tr}(\mfb_k)^2+2{\rm tr}(\mfb_k^2)\big\}-3{\rm tr}(\mfb_k)^4\Big\} \\
&\quad = \frac{C}{n^2}\SUML\big\{48{\rm tr}(\mfb_k^4)+12{\rm tr}(\mfb_k^2)^2\big\} \\
&\quad \leq \frac{C}{n^2}(\max_k\tilde{M}_k+1)^2L_n(1-\bar{\rho}_n)^{-4}1_{\{\bar{\rho}_n<1\}}+o_p(1)\to 0.
}

Moreover, Lemmas~\ref{Sn-lemma}, ~\ref{Sl-est-lemma},~\ref{Sl-est-lemma2} and~\ref{XAX-lemma} yield
\EQNN{&E_\Pi\bigg[\frac{C}{n^2}\SUML\bigg\{\bigg(\sum_{k'<k}\Delta^{(k')}X^c\bigg)^\top \PSINVS \SL \PSINVS \bigg(\sum_{k'<k}\Delta^{(k')}X^c\bigg)\bigg\}^2\bigg] \\
&\quad =\frac{C}{n^2}\SUML\sum_{k'_1,k'_2<k}\big\{|{\rm tr}(\PSINVS \SL \PSINVS S_{n,0}^{(k'_1)}){\rm tr}(\PSINVS \SL \PSINVS S_{n,0}^{(k'_2)})| \\
&\quad \quad \qquad \qquad +|{\rm tr}(\PSINVS \SL \PSINVS S_{n,0}^{(k'_1)}\PSINVS \SL \PSINVS S_{n,0}^{(k'_2)})|\big\} \\
&\leq \frac{C}{n^2}\SUML \big\{|{\rm tr}(\PSINVS \SL \PSINVS S_{n,0}^{(k-1)})^2|+|{\rm tr}((\PSINVS \SL \PSINVS S_{n,0}^{(k-1)})^2)|\big\}
+\frac{C}{n^2}L_n^3M^2\frac{\mcq_n\bar{\rho}_n^{\mcq_n}}{(1-\bar{\rho}_n)^2}1_{\{\bar{\rho}_n<1\}}+o_p(1)  \\
& =O_p\bigg(\frac{L_n}{n^2}\Big\{\max_k\tilde{M}_k\Big\}^2\bigg)+o_p(1)\TOP 0
}
as $n\to\infty$. 
Therefore, we have (\ref{mcx4-conv}).

~

Next, we show (\ref{mcx2-conv}).
Let $\mci_{i,j}^k=I_i\cap I_j\cap J^k$. Then, we obtain
\EQN{\label{mcx2-est-eq} \SUML E_k[\mcx_k^2]&=\frac{1}{n}\SUML \sum_{i_1,j_1}\sum_{i_2,j_2}[\PSINVS]_{i_1,j_1}[\PSINVS]_{i_2,j_2}
\int_{\mci_{i_1,i_2}^k}[\Sigma_{t,0}]_{\psi(i_1),\psi(i_2)}E_k[\Delta_{j_1,t} X^c \Delta_{j_2,t} X^c]dt \\
&=\frac{1}{n}\SUML \sum_{i_1,j_1}\sum_{i_2,j_2}[\PSINVS]_{i_1,j_1}[\PSINVS]_{i_2,j_2}
\int_{\mci_{i_1,i_2}^k}[\Sigma_{t,0}]_{\psi(i_1),\psi(i_2)}\int_{I_{j_1}\cap I_{j_2}\cap [0,t)}[\Sigma_{s,0}]_{\psi(j_1),\psi(j_2)}dsdt.
}

We can decompose 
\EQNN{\int_{\mci_{i_1,i_2}^k}[\Sigma_{t,0}]_{\psi(i_1),\psi(i_2)}\int_{I_{j_1}\cap I_{j_2}\cap [0,t)}[\Sigma_{s,0}]_{\psi(j_1),\psi(j_2)}dsdt 
 =\int_0^{T_n}F_{i_1,i_2}^k(t)\int_0^t F_{j_1,j_2}^k(s)dsdt + \sum_{k'<k}\mcf_{i_1,i_2}^k\mcf_{j_1,j_2}^{k'},
}
where $F_{ij}^k(t)=[\Sigma_{t,0}]_{\psi(i),\psi(j)}1_{\mci_{i,j}^k}(t)$, and $\mcf_{i,j}^k=\int_0^{T_n}F_{i,j}^k(t)dt$.
Moreover, switching the roles of $i_1,i_2$ and $j_1,j_2$, we obtain
\EQNN{&\sum_{i_1,j_1}\sum_{i_2,j_2}[\PSINVS]_{i_1,j_1}[\PSINVS]_{i_2,j_2}\int_0^{T_n}F_{i_1,i_2}^k(t)\int_0^tF_{j_1,j_2}^k(s)dsdt \\
&\quad =\sum_{i_1,j_1}\sum_{i_2,j_2}[\PSINVS]_{i_1,j_1}[\PSINVS]_{i_2,j_2}
\times \frac{1}{2}\bigg\{\int_0^{T_n}F_{i_1,i_2}^k(t)\int_0^tF_{j_1,j_2}^k(s)dsdt+\int_0^{T_n}F_{j_1,j_2}^k(t)\int_0^tF_{i_1,i_2}^k(s)dsdt\bigg\} \\
&\quad =\frac{1}{2}\sum_{i_1,j_1}\sum_{i_2,j_2}[\PSINVS]_{i_1,j_1}[\PSINVS]_{i_2,j_2}
\bigg\{\int_0^{T_n}F_{i_1,i_2}^k(t)\int_0^tF_{j_1,j_2}^k(s)dsdt+\int_0^{T_n}F_{i_1,i_2}^k(s)\int_s^{T_n}F_{j_1,j_2}^k(t)dtds\bigg\} \\
&\quad =\frac{1}{2}\sum_{i_1,j_1}\sum_{i_2,j_2}[\PSINVS]_{i_1,j_1}[\PSINVS]_{i_2,j_2}\mcf_{i_1,i_2}^k\mcf_{j_1,j_2}^k.
}

Therefore, we have
\EQN{\label{mcx2-est-eq2} \SUML E_k[\mcx_k^2]&=\frac{1}{2n}\SUML \sum_{i_1,j_1}\sum_{i_2,j_2}[\PSINVS]_{i_1,j_1}[\PSINVS]_{i_2,j_2}
\bigg\{\mcf_{i_1,i_2}^k\mcf_{j_1,j_2}^k+2\sum_{k'<k}\mcf_{i_1,i_2}^k\mcf_{j_1,j_2}^{k'}\bigg\} \\
&=\frac{1}{2n}\sum_{k,k'=1}^{L_n} \sum_{i_1,j_1}\sum_{i_2,j_2}[\PSINVS]_{i_1,j_1}[\PSINVS]_{i_2,j_2}\mcf_{i_1,i_2}^k\mcf_{j_1,j_2}^{k'} \\
&=\frac{1}{2n}\sum_{i_1,j_1}\sum_{i_2,j_2}[\PSINVS]_{i_1,j_1}[\PSINVS]_{i_2,j_2}\int_{I_{i_1}\cap I_{i_2}}[\Sigma_{t,0}]_{\psi(i_1),\psi(i_2)}dt
\int_{I_{j_1}\cap I_{j_2}}[\Sigma_{s,0}]_{\psi(j_1),\psi(j_2)}ds \\
&=\frac{1}{2n}{\rm tr}((\PSINVS S_{n,0})^2).
}

$\PSINVS S_{n,0}$ corresponds to $\hat{\mcd}(t)$ in the proof (p. 2993) of Proposition 10 of~\cite{ogi-yos14}.
Then by a similar step to the proof of Proposition 10 in~\cite{ogi-yos14}, we have (\ref{mcx2-conv}).

\begin{discuss}
{\colorr 
\EQNN{\hat{\mcd}=\sum_{p=0}^\infty \PS \bigg\{\MAT{cc}{B^1I_1 & 0 \\ 0 & B^2 I_2}\rho^{2p}
\MAT{cc}{(GG^\top)^p & -\rho (GG^\top)^pG \\ -\rho (G^\top G)^pG^\top & (G^\top G)^p}\MAT{cc}{B^1I_1 & 0 \\ 0 & B^2I_2}\bigg\}
\MAT{cc}{I & -\rho_\ast G \\ -\rho_\ast G^\top & I}.}

Lemma 10は${\rm tr}((GG^\top)^p)$のtightnessを使っているだけだからできる

CLTの分割$\check{s}_k$では(A3)のような評価は使わず，(\ref{mcx2-est-eq2})式で$S_{n,0}$に戻してから収束をみればいいので，(A3)で$\check{s}_k$分割を考えなくても問題ない．
}
\end{discuss}

\end{proof}

\begin{proposition}
Assume (A1)--(A4) and (A6). Then, $\Gamma_1$ is positive definite and 
\EQQ{\sqrt{n}(\HSN-\sigma_0)\TOD N(0,\Gamma_1^{-1})}
as $n\to\infty$.
\end{proposition}

\begin{proof}

Proposition~\ref{mcy1-est-prop}, (A6) and Remark 4 in~\cite{ogi-yos14} yield
\EQQ{\mcy_1(\sigma)\leq -c |\sigma-\sigma_0|^2
}
for some positive constant $c$.
Therefore, $\Gamma_1=\PS^2 \mcy_1(\sigma_0)$ is positive definite.

By Taylor's formula and the equation $\PS H_n^1(\HSN)=0$, we have
\EQNN{-\PS H_n^1(\sigma_0)&=\PS H_n^1(\hat{\sigma}_n)-\PS H_n^1(\sigma_0) \\
&=\int_0^1\PS^2H_n^1(\sigma_t)dt(\hat{\sigma}_n-\sigma_0) \\
&=\PS^2H_n^1(\sigma_0)(\hat{\sigma}_n-\sigma_0)+(\hat{\sigma}_n-\sigma_0)^\top \int_0^1(1-t)\PS^3H_n^1(\sigma_t)dt(\hat{\sigma}_n-\sigma_0),
}
where $\sigma_t=t\hat{\sigma}_n+(1-t)\sigma_0$.

Therefore, we obtain
\EQ{\label{HSN-error-eq} \sqrt{n}(\HSN-\sigma_0)=\bigg\{-\frac{1}{n}\PS^2H_n^1(\sigma_0)-\frac{1}{n}\int_0^1(1-t)\PS^3H_n^1(\sigma_t)dt(\HSN-\sigma_0)\bigg\}^{-1}\cdot \frac{1}{\sqrt{n}}\PS H_n^1(\sigma_0).}

Since Proposition~\ref{H1n-conv-prop} yields
\EQQ{-\frac{1}{n}\PS^2 H_n^1(\sigma_0)\TOP -\PS^2\mcy_1(\sigma_0)=\Gamma_1,}
and Sobolev's inequality yields that
\EQQ{\bigg\{\sup_\sigma\bigg|\frac{1}{n}\PS^3H_n^1(\sigma)\bigg|\bigg\}_n}
is $P$-tight, we conclude
\EQ{\label{HSN-conv} \sqrt{n}(\HSN-\sigma_0)\TOD N(0,\Gamma_1^{-1}).}
\end{proof}

\subsection{Consistency of $\HTN$}\label{HTN-consis-subsection}

\begin{discuss}
{\colorr
\EQQ{H_n^2(\theta)=-\frac{1}{2}\bar{X}(\theta)^\top E S_n^{-1}(\HSN)E\bar{X}(\theta)+\sum_{i,l}\log f_\theta^l(\Delta_i^lX) 1_{\{|\Delta_i^l X|>c_n\}}-nh_n(\lambda^1(\theta)+\lambda^2(\theta)).}
}
\end{discuss}

Let
\EQQ{\mcy_2(\theta)=\lim_{T\to\infty}\frac{1}{T}\int_0^T\SUMP \bigg\{-\frac{1}{2}\sum_{l=1}^2f_p^{ll}\rho_{t,0}^{2p}\phi_{l,t}^2+f_p^{12}\rho_{t,0}^{2p+1}\phi_{1,t}\phi_{2,t}\bigg\}dt,}
which exists under (A1), (A3) and (A5).
\begin{proposition}
Assume (A1)--(A6). Then,
\EQ{\label{H2n-convergence} \sup_{\theta\in\Theta_2}\big|(nh_n)^{-1}\PT^k(H_n^2(\theta)-H_n^2(\theta_0))-\PT^k\mcy_2(\theta)\big|\TOP 0}
as $n\to\infty$ for $k\in \{0,1,2,3\}$.
\end{proposition}

\begin{proof}

Lemma~\ref{Sn-lemma} yields
\EQN{\label{drift-est-eq} &E_\Pi\big[(\Delta V(\theta)^\top \PS^m S_{n,0}^{-1}\Delta X^c)^2\big] \\
&\quad =\sum_{i_1,j_1}\sum_{i_2,j_2}[\PS^m S_{n,0}^{-1}]_{i_1,j_1}[\PS^m S_{n,0}^{-1}]_{i_2,j_2}\Delta_{i_1} V(\theta) \Delta_{i_2} V(\theta)E_\Pi[\Delta_{j_1}X^c \Delta_{j_2}X^c] \\
&\quad =\sum_{i_1,j_1}\sum_{i_2,j_2}[\PS^m S_{n,0}^{-1}]_{i_1,j_1}[\PS^m S_{n,0}^{-1}]_{i_2,j_2}\Delta_{i_1} V(\theta) \Delta_{i_2} V(\theta)[S_{n,0}]_{j_1,j_2} \\
&\quad \leq C|\mcd^{-1/2}\Delta V(\theta)|^2\lVert \mcd^{1/2}\PS^m S_{n,0}^{-1}\mcd^{1/2}\rVert^2\lVert \mcd^{-1/2}S_{n,0}\mcd^{-1/2}\rVert \\
&\quad \leq C(1-\bar{\rho}_n)^{-2m-2}\sum_i|I_i| \leq Cnh_n(1-\bar{\rho}_n)^{-2m-2}
}
on $\{\bar{\rho}_n<1\}$.

Since 
\EQQ{E_\Pi[|\mcd^{-1/2}\Delta X|^2]=\sum_i \frac{E_\Pi[|\Delta_i X|^2]}{|I_i|}\leq Cn,}
\EQQ{\bar{X}(\theta)^\top S_n^{-1}(\HSN)\bar{X}(\theta) -\Delta X^\top S_n^{-1}(\HSN)\Delta X 
=-\Delta V(\theta)^\top S_n^{-1}(\HSN)(2\Delta X - \Delta V(\theta)),}
and
\EQ{\label{Sn-exp} S_n^{-1}(\HSN)=S_{n,0}^{-1}+(\HSN-\sigma_0)\PS S_{n,0}^{-1}+\int_0^1(1-u)\PS^2S_n^{-1}(u\HSN+(1-u)\sigma_0)du (\HSN-\sigma_0)^2,}
(\ref{HSN-conv}), Lemma~\ref{Sn-lemma} and a similar estimate to (\ref{DeltaV2}) imply
\EQN{\label{drift-consis-eq1} &\sup_\theta\big|\bar{X}(\theta)^\top S_n^{-1}(\HSN)\bar{X}(\theta) -\Delta X^\top S_n^{-1}(\HSN)\Delta X
+\Delta V(\theta)^\top \big\{S_{n,0}^{-1}+(\HSN-\sigma_0)\PS S_{n,0}^{-1}\big\}(2\Delta X - \Delta V(\theta))\big| \\
&\quad =O_p((n^{-1/2})^2\cdot \sqrt{n}\cdot \sqrt{nh_n})=o_p(\sqrt{nh_n}).
}

Thanks to (\ref{drift-est-eq}) and Lemma~\ref{Sn-lemma}, we have
\EQN{\label{drift-consis-eq2} &\sup_\theta |\Delta V(\theta)^\top \big\{(\HSN-\sigma_0)\PS S_{n,0}^{-1}\big\}(2\Delta X - \Delta V(\theta))| \\
&\quad =\sup_\theta |\Delta V(\theta)^\top \big\{(\HSN-\sigma_0)\PS S_{n,0}^{-1}\big\}(2\Delta X^c + 2\Delta V(\theta_0) - \Delta V(\theta))| \\
&\quad \leq  \sup_\theta |2\Delta V(\theta)^\top\big\{(\HSN-\sigma_0)\PS S_{n,0}^{-1}\big\}\Delta X^c|
+C\sup_\theta|\mcd^{-1/2}\Delta V(\theta)|^2\lVert \mcd^{1/2}\PS S_{n,0}^{-1}\mcd^{1/2}\rVert |\HSN-\sigma_0| \\
&\quad \leq  \sup_\theta |2\Delta V(\theta)^\top\big\{(\HSN-\sigma_0)\PS S_{n,0}^{-1}\big\}\Delta X^c|+O_p(nh_n)\cdot O_p(n^{-1/2}).
}

For $k\in\{0,1\}$ and $q\geq 1$, the \BDG, Lemma~\ref{Sn-lemma} and a similar estimate to (\ref{DeltaV2}) yield
\EQN{\label{drift-consis-eq6} &\sup_\theta E_\Pi[|\PT^k \Delta V(\theta)^\top \PS S_{n,0}^{-1} \Delta X^c|^q]^{1/q} \\
&\quad \leq C_q\sup_\theta \SUMLL E_\Pi\bigg[\bigg|\sum_i[\PSINVS \PT^k\Delta V(\theta)]_{i+(l-1)M_1}\Delta_i^l X^c\bigg|^q\bigg]^{1/q} \\
&\quad \leq  C_q \sup_\theta \SUMLL \bigg(\sum_i [\PSINVS \PT^k\Delta V(\theta)]_{i+(l-1)M_1}^2|I_i^l|\bigg)^{1/2} \\
&\quad = C_q\sup_\theta \big(\PT^k\Delta V(\theta)^\top \PSINVS \mcd \PSINVS \PT^k\Delta V(\theta)\big)^{1/2} \\
&\quad \leq C_q\sqrt{nh_n}.
}
Together with (\ref{drift-consis-eq2}) and Sobolev's inequality, we have
\EQ{\label{drift-consis-eq5} \sup_\theta  |\Delta V(\theta)^\top \big\{(\HSN-\sigma_0)\PS S_{n,0}^{-1}\big\}(2\Delta X - \Delta V(\theta))|=o_p(\sqrt{nh_n}).}
Then, (\ref{drift-consis-eq1}) and (\ref{drift-consis-eq5}) yield
\EQN{\label{drift-consis-eq3} &\sup_\theta \big|\bar{X}(\theta)^\top S_n^{-1}(\HSN)\bar{X}(\theta) -\Delta X^\top S_n^{-1}(\HSN)\Delta X +2\Delta V(\theta)^\top  S_{n,0}^{-1}\Delta X^c +\Delta V(\theta)^\top S_{n,0}^{-1}(2\Delta V(\theta_0) - \Delta V(\theta))\big| \\
&\quad  =o_p(\sqrt{nh_n}).
}
Together with (\ref{drift-est-eq}), we obtain
\EQN{\label{drift-consis-eq4} &\sup_\theta \bigg|H_n^2(\theta)-H_n^2(\theta_0) \\
&\qquad -\bigg\{\Delta (V(\theta)-V(\theta_0))^\top S_{n,0}^{-1}\Delta X^c+\frac{1}{2}\Delta V(\theta)^\top S_{n,0}^{-1}(2\Delta V(\theta_0) - \Delta V(\theta))-\frac{1}{2}\Delta V(\theta_0)^\top S_{n,0}^{-1}\Delta V(\theta_0)\bigg\}\bigg| \\
&\quad =O_p(\sqrt{nh_n}),
}
and hence, similar estimates to (\ref{drift-consis-eq6}), we have
\EQQ{\sup_\theta \bigg|H_n^2(\theta)-H_n^2(\theta_0)+\frac{1}{2}\Delta (V(\theta) - V(\theta_0))^\top S_{n,0}^{-1}\Delta (V(\theta) - V(\theta_0))\bigg|=O_p(\sqrt{nh_n}).}

\begin{discuss}
{\colorr 
\EQQ{H_n^2(\theta)=-\frac{1}{2}\bar{X}(\theta)^\top E S_n^{-1}(\HSN)E\bar{X}(\theta)+\sum_{i,j}\log f_\theta^l(\Delta_i^lX) 1_{\{|\Delta_i^l X|>c_n\}}-nh_n(\lambda^1(\theta)+\lambda^2(\theta)).}
}
\end{discuss}

Then, (\ref{InvS-eq}) yields
\EQNN{&\Delta (V(\theta) - V(\theta_0))^\top S_{n,0}^{-1}\Delta (V(\theta) - V(\theta_0)) \\
&\quad =\Delta (V(\theta) - V(\theta_0))^\top \tilde{\mcd}^{-1/2}(\sigma_0)\sum_{p=0}^\infty \MAT{cc}{(\TILG\TILG^\top)^p & -(\TILG\TILG^\top)^p\TILG \\ -(\TILG^\top \TILG)^p\TILG^\top & (\TILG^\top \TILG)^p}\tilde{\mcd}^{-1/2}(\sigma_0)\Delta (V(\theta) - V(\theta_0)) \\
&\quad =\SUMP \SUMK \dot{\rho}_{k,0}^{2p}\bigg\{\SUMLL (\phi_{l,s_{k-1}})^2\dot{\mfi}_{k,l}^\top \dot{\mca}_{k,p}^l\dot{\mfi}_{k,l}
-2\dot{\rho}_{k,0}\phi_{1,s_{k-1}}\phi_{2,s_{k-1}}\dot{\mfi}_{k,1}^\top \dot{\mca}_{k,p}^1 G_k\dot{\mfi}_{k,2}\bigg\}+nh_ne_n,
}
where $\dot{\mfi}_{k,l}=\mce_{(k)}^l\dot{\mfi}_l$.
Together with (A3), (A5) and (\ref{drift-consis-eq4}), we obtain
\EQ{\label{H2n-convergence} \sup_\theta\big|(nh_n)^{-1}(H_n^2(\theta)-H_n^2(\theta_0)) - \mcy_2(\theta)\big|\TOP 0}
as $n\to\infty$.
Similar estimates for $(nh_n)^{-1}\PT^k(H_n^2(\theta)-H_n^2(\theta_0))$ $(k\in\{0,1,2,3,4\}$ yield the conclusion.

\end{proof}

\begin{proposition}
Assume (A1)--(A6). Then, $\HTN\TOP \theta_0$ as $n\to\infty$.
\end{proposition}

\begin{proof}

By Lemma~\ref{Sn-lemma}, we have
\EQ{\mcd^{1/2}S_{n,0}^{-1}\mcd^{1/2}\geq \lVert \mcd^{-1/2}S_{n,0}\mcd^{-1/2}\rVert^{-1} \mce_M\geq C\mce_M.}
Therefore, together with (\ref{uniform-conti-eq}) and (\ref{H1n-est-eq1}), we obtain
\EQN{-\frac{1}{2}\Delta (V(\theta) - V(\theta_0))^\top S_{n,0}^{-1}\Delta (V(\theta) - V(\theta_0)) 
& \leq -C\Delta (V(\theta) - V(\theta_0))^\top \mcd^{-1}\Delta (V(\theta) - V(\theta_0)) \\
&= -C\SUMK \sum_{l=1}^2\sum_i \phi_{l,s_{k-1}}^2|I_i^l\cap J_k|+nh_ne_n \\
& = -C\int_0^{T_n}\sum_{l=1}^2\phi_{l,t}^2dt+nh_ne_n
}
Hence, we have
\EQ{\label{Y2-iden2} \mcy_2(\theta)\leq -C\lim_{T\to\infty}\bigg(\frac{1}{T}\int_0^T(\phi_{1,t}^2+\phi_{2,t}^2)dt\bigg).}

Assumption (A6) yields that for any $\theta\in\Theta$, 
\EQ{\label{Y2-iden} \mcy_2(\theta)\leq 0, \quad {\rm and} \quad \mcy_2(\theta)=0 \quad {\rm if~and~only~if} \quad \theta=\theta_0.}

(\ref{H2n-convergence}), (\ref{Y2-iden}) together with a similar estimates to (\ref{sigma-consis-eq}), we have the conclusion.

\end{proof}

\subsection{Asymptotic normality of $\HTN$}\label{HTN-clt-subsection}

\noindent
{\bf Proof of Theorem~\ref{main-thm}.}

A similar estimate to (\ref{drift-consis-eq3}) yields
\EQNN{\PT H_n^2(\theta_0)&=(\PT \Delta V(\theta_0))^\top  S_n^{-1}(\HSN)\bar{X}(\theta_0) \\
&= \PT\Delta V(\theta_0)^\top  S_{n,0}^{-1}\Delta X^c+\frac{1}{2}\PT \Delta V(\theta_0)^\top S_{n,0}^{-1}\Delta V(\theta_0) -\frac{1}{2}\Delta V(\theta_0)^\top S_{n,0}^{-1}\PT\Delta V(\theta_0) 
+o_p(\sqrt{nh_n}) \\
&=\PT\Delta V(\theta_0)^\top  S_{n,0}^{-1}\Delta X^c+o_p(\sqrt{nh_n}).
}

Let 
\EQQ{\dot{\mcx}_k=\frac{1}{\sqrt{nh_n}} \PT\Delta V(\theta_0) S_{n,0}^{-1}\Delta^{(k)} X^c
}
for $1\leq k\leq L_n$. 
Then, we have
\EQ{\label{H2n-eq} (nh_n)^{-1/2}\PT H_n^2(\theta_0)= \SUML \dot{\mcx}_k +o_p(1).}

Lemma~\ref{Sn-lemma} yields
\EQNN{\SUML E_k[\dot{\mcx}_k^4]&=\frac{3}{n^2h_n^2}\SUML \big\{\PT \Delta V(\theta_0)^\top S_{n,0}^{-1}S_{n,0}^{(k)}S_{n,0}^{-1}\PT \Delta V(\theta_0)\big\}^2 \\
&\leq \frac{C}{n^2h_n^2}|\mcd^{-1/2}\Delta \PT V(\theta_0)|^2\lVert \mcd^{1/2}S_{n,0}^{-1}\mcd^{1/2}\rVert^2 \SUML \lVert
\mcd^{-1/2}S_{n,0}^{(k)}\mcd^{-1/2}\rVert \leq \frac{CL_n}{nh_n} \TOP 0.
}
Moreover, simple calculation shows that
\EQNN{\SUML E_k[\dot{\mcx}_k^2]&=\frac{1}{nh_n}\SUML \sum_{i_1,j_1}\sum_{i_2,j_2}[S_{n,0}^{-1}]_{i_1,j_1}[S_{n,0}^{-1}]_{i_2,j_2} \Delta_{i_1} \PT V(\theta_0)\Delta_{i_2} \PT V(\theta_0)[S_{n,0}^{(k)}]_{j_1,j_2} \\
&=\frac{1}{nh_n}\Delta \PT V(\theta_0)^\top S_{n,0}^{-1}S_{n,0}S_{n,0}^{-1}\Delta \PT V(\theta_0) \\
&=\frac{1}{nh_n}\SUMP \SUMK \dot{\rho}_{k,0}^{2p}\bigg\{\SUMLL \PT \phi_{l,s_{k-1}}^2(\theta_0)\mfi_l^\top \mca_{k,p}^l\mfi_l
-2\dot{\rho}_{k,0}\PT\phi_{1,s_{k-1}}\PT \phi_{2,s_{k-1}}(\theta_0)\mfi_1^\top \mca_{k,p}^1G\mfi_2\bigg\}+e_n \\
&\TOP \Gamma_2.
}

Therefore, (\ref{H2n-eq}) and the martingale central limit theorem (Corollary 3.1 and the remark after that in Hall and Heyde~\cite{hal-hey80}) yield
\EQ{\label{DXL-conv} (nh_n)^{-1/2}\PT H_n^2(\theta_0)=\SUML \dot{\mcx}_k+o_p(1)\TOD N(0,\Gamma_2).}

By (\ref{Y2-iden2}) and (A5), there exists a positive constant $c$ such that $\mcy_2(\theta)\leq -c|\theta-\theta_0|^2$. 
Then, $\Gamma_2=\PT^2\mcy_2(\theta_0)$ is positive definite.

Therefore, a similar estimate to Section~\ref{HSN-clt-subsection}, $P$-tightness of 
$\{(nh_n)^{-1}\sup_\theta|\PT^3H_n^2(\theta)|\}_n$, and the equation $-(nh_n)^{-1}\PT^2H_n^2(\theta_0)\TOP \Gamma_2$ yield
\EQQ{\sqrt{T_n}(\HTN-\theta_0)\TOD N(0,\Gamma_2^{-1}).}

(\ref{HSN-error-eq}) and a similar equation for $\sqrt{nh_n}(\HTN-\theta_0)$ yield
\EQN{\label{joint-conv-eq} (\sqrt{n}(\HSN-\sigma_0),\sqrt{T_n}(\HTN-\theta_0))
&= (n^{-1/2}\Gamma_1^{-1}\PS H_n^1(\sigma_0),T_n^{-1/2}\Gamma_2^{-1}\PT H_n^2(\theta_0))+o_p(1) \\
&= \sum_{k=1}^{L_n}(\Gamma_1^{-1}\mcx_k,\Gamma_2^{-1}\dot{\mcx}_k)+o_p(1).
}
Then, since $\SUML E_k[\mcx_k\dot{\mcx}_k]=0$, we obtain 
\EQQ{(\sqrt{n}(\HSN-\sigma_0), \sqrt{nh_n}(\HTN-\theta_0))\TOD N(0,\Gamma^{-1}).}

\qed


\subsection{Proofs of the results in Sections~\ref{LAN-subsection} and~\ref{suff-cond-subsection}}\label{suff-cond-proof-subsection}

\noindent
{\bf Proof of Theorem~\ref{LAN-thm}.}

Let 
\EQQ{H_n(\sigma,\theta)=-\frac{1}{2}\bar{X}(\theta)^\top S_n^{-1}(\sigma)\bar{X}(\theta)
-\frac{1}{2}\log \det S_n(\sigma).}
Then, we have
\EQNN{& H_n(\sigma_u,\theta_u) \\
&\quad =\int_0^1 \partial_\alpha H_n(\sigma_{tu},\theta_{tu})dt \epsilon_n u \\
&\quad =u^\top \epsilon_n \partial_\alpha H_n(\sigma_0,\theta_0) 
+\frac{1}{2}u^\top \epsilon_n \partial_\alpha^2 H_n(\sigma_0,\theta_0)\epsilon_n u \\
&\quad \quad +\sum_{i,j,k}\int_0^1\frac{(1-s)^2}{2}\partial_{\alpha_i}\partial_{\alpha_j}\partial_{\alpha_k} H_n(\sigma_{su},\theta_{su})ds [\epsilon_n u]_i[\epsilon_n u]_j[\epsilon_n u]_k.
}
By  similar arguments to Propositions~\ref{H1n-conv-prop} and~\ref{Hn1-st-conv-prop}, and Sections~\ref{HTN-consis-subsection} and~\ref{HTN-clt-subsection}, we obtain
\EQNN{\sum_{i,j,k}\int_0^1\frac{(1-s)^2}{2}\partial_{\alpha_i}\partial_{\alpha_j}\partial_{\alpha_k} H_n(\sigma_{su},\theta_{su})ds [\epsilon_n u]_i[\epsilon_n u]_j[\epsilon_n u]_k&\TOP 0, \\
\epsilon_n \partial_\alpha H_n(\sigma_0,\theta_0) &\TOD N(0,\Gamma), \\
\epsilon_n \partial_\alpha^2 H_n(\sigma_0,\theta_0) \epsilon_n &\TOP \Gamma.
}	
Therefore, we have the desired conclusion.

\qed

\begin{discuss}
{\colorr $A^{p,+}$や$A^{p,-}$に対応するものは$(s_{j-1},s_j]$からはみ出したところも考えるか．

OgiYos14ではおそらく$t_k=k/[b_n]$, $b_n=n$. $X'_k$は$b_n=[h_n^{-1}]$をかけていて，(31)式は$b_n^{-1}X'_k$に対応するから$h_n$がつく．
$X'_k$のモーメントは(B1)で保証される．}
\end{discuss}

\noindent
{\bf Proof of Proposition~\ref{A3-suff-prop}.}

The proof is similar to the proof of Proposition 6 in~\cite{ogi-yos14}.

$P$-tightness of $\{h_nM_{l,q_n+1}\}_{n=1}^\infty$ imediately follows from (B1-$1$).
Fix $1\leq j\leq q_n$. In the proof of Proposition 6 in Section 7.5 of~\cite{ogi-yos14},
we definte $b_n=h_n^{-1}$, $t_k=s_{j-1}+k[h_n^{-1}]^{-1}(s_j-s_{j-1})$ $(0\leq k\leq [h_n^{-1}]$, 
and $X'_k={\rm tr}(\mce_{(j,k)}^1(GG^\top)^p)1_{A_{k,b_n^{\delta'}}^p}-E[{\rm tr}(\mce_{(j,k)}(GG^\top)^p)1_{A_{k,b_n^{\delta'}}^p}]$,
where $\mce_{(j,k)}^l$ be an $M_l\times M_l$ matrix satisfying $[\mce_{(j,k)}^l]_{ij}=1$ if $i=j$ and $\sup I_i^l\in (t_{k-1},t_k]$,
and otherwise $[\mce_{(j,k)}^l]_{ij}=0$.

Then, similarly to (31) in~\cite{ogi-yos14}, there exists $\eta>0$ such that for any $q\geq 4$, there exists $C_q>0$ that does not depend on $k$ such that
\EQQ{E\big[\big|h_n{\rm tr}(\mce_{(j,k)}(GG^\top)^p)-E[h_n{\rm tr}(\mce_{(j,k)}(GG^\top)^p)]\big|^q\big]
\leq C(p+1)^{q-1}h_n^{q\eta}.}
Therefore, by setting sufficiently large $q$ so that $nh_n^{1+q\eta}\to 0$, we have
\EQNN{&E\bigg[\max_{1\leq k\leq q_n}\big|h_n{\rm tr}(\mce_{(j,k)}(GG^\top)^p)-E[h_n{\rm tr}(\mce_{(j,k)}(GG^\top)^p)]\big|^q\bigg] \\
&\quad \leq E\bigg[\sum_{k=1}^{q_n}\big|h_n{\rm tr}(\mce_{(j,k)}(GG^\top)^p)-E[h_n{\rm tr}(\mce_{(j,k)}(GG^\top)^p)]\big|^q\bigg] \\
&\quad =O(nh_n \cdot h_n^{q\eta})\to 0.
}
Together with the assumptions, we obtain the conclusion.

\begin{discuss}
{\colorr $q>2$と$\eta\in (0,1)$に対して，(A3$'$-$q,\eta$)は(A3$'$)をimplyする．
$\int_0^{T_n}a_p(t)dt<\infty$はProp 6の証明で$f$のHolder連続で積分を近似するときに必要で今は不要．
}
\end{discuss}

\qed

\noindent
{\bf Proof of Proposition~\ref{A5-suff-cond-prop}.}

We use the proof of Proposition 6 in~\cite{ogi-yos14} again.
We define $b_n$ and $t_k$ the same as the previous proposition, and define
\EQQ{X'_k=[h_n]^{-1}\mfi_1^\top\mce_{(j,k)}(GG^\top)^p\mfi_11_{A_{k,b_n^{\delta'}}^p}-E[[h_n]^{-1}\mfi_1^\top\mce_{(j,k)}(GG^\top)^p\mfi_11_{A_{k,b_n^{\delta'}}^p}].}
Then, similarly to (31) in the proof, there exists $\eta>0$ such that for any $q\geq 4$, there exists $C_q>0$ such that
\EQQ{E\Big[\big|\mfi_1^\top\mce_{(j,k)}(GG^\top)^p\mfi_1-E[\mfi_1^\top\mce_{(j,k)}(GG^\top)^p\mfi_1]\big|^q\Big]\leq C_q(p+1)^{q-1}h_n^{q\eta}.}
Together with the assumptions and similar estimates for $\mfi_1\EK^1(GG^\top)^pG\mfi_2$ and $\mfi_2\EK^2(G^\top G)^p\mfi_2$, 
we obtain the conclusion.

\qed

\noindent
{\bf Proof of Proposition~\ref{a11-suff-cond-prop}.}

We can show the results by a similar approach to the proof of Proposition 9 in~\cite{ogi-yos14}.
Under (B2-$q$), $P(\mcn_{t+Nh_n}-\mcn_t=0)$ is small enough to estimate the denominator of
\EQQ{\sum_{i,j}\frac{|I_i\cap I_j|^2}{|I_i||I_j|}}
for sufficiently large $n$. Then, we obtain estimates for the numerator by using an inequality
$x_1^2+\cdots +x_n^2\geq R^2/n$ when $x_1+\cdots +x_n=R$.

\qed

\noindent
{\bf Proof of Lemma~\ref{A4A5-suff-cond-prop}.}

We only show 
\EQQ{\max_{1\leq k\leq q_n}|h_nE[{\rm tr}(\EK^1(GG^\top)^p)]-a_p^1(s_k-s_{k-1})|\to 0.}
The other results are similarly obtained.

(\ref{mixing-condition}) is satisfied because $\alpha_k^n\leq c_1e^{-c_2k}$ for some positive constants $c_1$ and $c_2$.

Let $\bar{\tau}_i^l$ be $i$-th jump time of $\BN^l$. Then, we have $\SI=h_n\bar{\tau}_i^l$.
Let $\bar{G}$ be a matrix with infinity side defined by
\EQQ{[\bar{G}]_{ij}=\frac{|[\bar{\tau}_{i-1}^1,\bar{\tau}_i^1)\cap [\bar{\tau}_{j-1}^2,\bar{\tau}_j^2)|}{\sqrt{\bar{\tau}_i^1-\bar{\tau}_{i-1}^1}\sqrt{\bar{\tau}_j^2-\bar{\tau}_{j-1}^2}}}
for $i,j\geq 1$.

For $k\in\mbbn$, let
\EQQ{\mfg_k^p=\sum_{i;\bar{\tau}_{i-1}^1\in [k-1,k)}[(\bar{G}\bar{G}^\top)^p]_{ii}, \quad 
\mfg_k^{n,p}=\sum_{i;S_{i-1}^{n,1}\in [(k-1)h_n,kh_n)}[(GG^\top)^p]_{ii}.}
The following idea is based on Section 7.5 of~\cite{ogi-yos14}. Roughly speaking, if there are sufficient observations around the interval $[k-1,k)$, we can apply mixing property of $\BN_t^{n,l}$ to $\mfg_k^p$.
On the following sets $A_{k,r}^p$ and $\bar{A}_{k,r}^p$, we have sufficient observations of $\mcn^{n,l}$ and $\BN^l$.
Let $\bar{\Delta}_{j,t}^r U=U_{t+rj}-U_{t+r(j-1)}$ for a stochastic process $(U_t)_{t\geq 0}$, and let

\EQN{A_{k,r}^p&=\bigcap_{l=1,2}\bigg\{\bigcap_{\substack{1\leq j\leq 2p+1 \\ t_k+rjh_n\leq T_n}}\{\bar{\Delta}_{j,t_k}^{rh_n}\mcn^{n,l}>0\} 
\cap \bigcap_{\substack{-2p\leq j\leq 0 \\ t_{k-1}+r(j-1)h_n\geq 0}}\{\bar{\Delta}_{j,t_{k-1}}^{rh_n}\mcn^{n,l}>0\}\bigg\}, \\
\bar{A}_{k,r}^p&=\bigcap_{l=1,2}\bigg\{\bigcap_{1\leq j\leq 2p+1}\{\bar{\Delta}_{j,k}^{r}\BN^l>0\} 
\cap \bigcap_{\substack{-2p\leq j\leq 0 \\ k-1+r(j-1)\geq 0}}\{\bar{\Delta}_{j,k-1}^r\BN^l>0\}\bigg\}.
}

Then, we obtain
\EQNN{E[\mfg_k^p1_{\bar{A}_{k,r}^p}]&=E[\mfg_{k'}^p1_{\bar{A}_{k',r}^p}] \quad {\rm if} \quad k\wedge k'\geq rp+1, \\
E[\mfg_k^{n,p}1_{A_{k,r}^p}]&=E[\mfg_{k'}^{n,p}1_{A_{k',r}^p}] \quad {\rm if} \quad rp+1\leq k,k'\leq n-rp.
}
We also have $P((\bar{A}_{k,r}^p)^c)\leq C(p+1)r^{-q}$ by (B2-$q$).
For any $\epsilon>0$, there exists $r>0$ such that
\EQ{\label{barAc-est} P((\bar{A}_{k,r}^p)^c)<\epsilon/2.}
Therefore, $\{E[\mfg_k^p]\}_k$ is a Cauchy sequence, and hence, the limit $a_p^1=\lim_{k\to\infty}E[\mfg_k^p]$ exists for $p\in \mbbn$. 
Moreover, we see existence of 
\EQQ{a_0^l=\lim_{k\to\infty}E[\bar{\mcn}_k^l-\bar{\mcn}_{k-1}^l]=E[\bar{\mcn}_1^l-\bar{\mcn}_0^l]}
for $l\in \{1,2\}$.

Furthermore, for any $\epsilon>0$, there exists $r>0$ such that
$P((\bar{A}_{k,r}^p)^c)<\epsilon$ and $|E[\mfg_k^p]-a_p^1|<\epsilon$ for $k\geq [rp]$.
Let $r_j=[h_n^{-1}s_j]$. Then, since $|\mfg_k^{n,p}|\leq \sum_{i;\SIm \in ((k-1)h_n,kh_n]}1\leq E[\bar{\mcn}_1^1]$ and
\EQQ{\sup I_i^l\in (s_{j-1},s_j] \quad \Longleftrightarrow \quad \bar{\tau}_i^l\in (h_n^{-1}s_{j-1},h_n^{-1}s_j],}
the \CS ~yields
\EQNN{&|h_n(s_j-s_{j-1})^{-1}E[{\rm tr}(\mce_{(j)}(GG^\top)^p)]-a_p^1| \\
&\quad \leq \bigg|h_n(s_j-s_{j-1})^{-1}\sum_{k=r_{j-1}+1}^{r_j}\mfg_k^{n,p}-a_p^1\bigg|+2h_n(s_j-s_{j-1})^{-1}E[\BN_1^1] \\
&\quad \leq \bigg|\frac{1}{r_j-r_{j-1}}\sum_{k=r_{j-1}+1}^{r_j}\mfg_k^{n,p}-a_p^1\bigg|+Ch_n(s_j-s_{j-1})^{-1} \\
&\quad \leq \frac{1}{r_j-r_{j-1}}\sum_{k=r_{j-1}+1}^{r_j}\big|E[\mfg_k^{n,p}1_{A_{k,h}^p}]+E[\mfg_k^{n,p}1_{(A_{k,h}^p)^c}]-a_p^1\big|+Ch_n(s_j-s_{j-1})^{-1} \\
&\quad \leq \frac{1}{r_j-r_{j-1}}\sum_{k=r_{j-1}+1}^{r_j}\big(\big|E[\mfg_k^p]-a_p^1\big|+2E[(\bar{\mcn}_1^1)^2]^{1/2}\sqrt{\epsilon}\big)+Ch_n(s_j-s_{j-1})^{-1} \\
&\quad \leq \epsilon+2E[(\bar{\mcn}_1^1)^2]^{1/2}\sqrt{\epsilon}+Ch_n(s_j-s_{j-1})^{-1}.
}
we replace the minimum $r_{j-1}+1$ of the summation range of $k$ with $r_{j-1}+[rp]+2$ when $j=1$,
and replace the maximum $r_j$ with $r_j-[rp]-1$ when $j=q_n$.
Then, the conclusion.

\qed

\noindent 
{\bf Acknowledgements}
On behalf of all authors, the corresponding author states that there is no conflict of interest.

\begin{small}
\bibliographystyle{abbrv}
\bibliography{referenceLibrary_mathsci,referenceLibrary_other}
\end{small}

\renewcommand{\thesection}{\Alph{section}}
\setcounter{section}{0}
\section{Appendix}

\begin{lemma}\label{XAX-lemma}
Let $m\in \mbbn$. Let $V$ be an $m\times m$ symmetric, positive definite matrix and $A$ be a $m\times m$ matrix. 
Let $X$ be a random variable following $N(0,V)$. Then
\EQNN{E[(X^\top AX)^2]&={\rm tr}(AV)^2+2{\rm tr}((AV)^2), \\
E[(X^\top AX)^3]&={\rm tr}(AV)^3+6{\rm tr}(AV){\rm tr}((AV)^2)+8{\rm tr}((AV)^3), \\
E[(X^\top AX)^4]&={\rm tr}(AV)^4+12{\rm tr}(AV)^2{\rm tr}((AV)^2)+12{\rm tr}((AV)^2)^2+32{\rm tr}(AV){\rm tr}((AV)^3)+48{\rm tr}((AV)^4).
}
\end{lemma}

\begin{proof}
We only show the result for $E[(X^\top AX)^4]$.
Let $U$ be an orthogonal matrix and $\Lambda$ be a diagonal matrix satisfying $UVU^\top =\Lambda$.
Then, we have $UX\sim N(0,\Lambda)$, and 
\EQQ{E\bigg[\prod_{i=1}^8[UX]_{j_i}\bigg]=\sum_{(l_{2q-1},l_{2q})_{q=1}^4}\prod_{q=1}^4[\Lambda]_{l_{2q-1},l_{2q}},}
where the summation of $(l_{2q-1},l_{2q})_{q=1}^4$ is taken over all disjoint pairs of $\{j_1,\cdots j_8\}$.
Then, by setting $B=UAU^\top$, we have
\EQQ{E[(X^\top AX)^4]=\sum_{j_1,\cdots, j_8}\sum_{(l_{2q-1},l_{2q})_{q=1}^4}\prod_{p=1}^4[B]_{j_{2p-1},j_{2p}}\prod_{q=1}^4[\Lambda]_{l_{2q-1},l_{2q}},}
which yields the conclusion.
\end{proof}

\begin{discuss}
{\colorr 一次元の場合の標準正規分布のモーメント$E[X^{2p}]=(2p-1)!!$を見ると係数は正しそう．

 $E[(X^\top AX)^3]$の右辺第２項は閉じたペアが１つとクロスしたペアが１つで，係数は閉じたペアの選び方が３通り×クロスしたペアの中で連結の仕方が２通りあるので６通り，第３項は３ペアの連結で，連結の仕方は$j_1$と連結する者が４通りと，その中で$j_2$と連結する者が２通りなので８通り．

 $E[(X^\top AX)^4]$の右辺第２項は閉じたペア２つとクロスしたペアが１つで，閉じたペアの選び方が６通り×クロスしたペアの連結の仕方が２通りで１２通り．
 右辺第３項はクロスしたペアが２つで，クロスの選び方は$6/2=3$通り（(1,3)と(2,4)は同一視される)でクロスの連結の仕方がそれぞれ２通りなので12通り．
 右辺第４項は閉じたペア１つと３ペアの連結は，閉じたペアの選び方４通りと３ペアの連結が４×２通りなので32通り．
 右辺第５項は４ペアの連結が６×４×２＝４８通り．
}
\end{discuss}


\begin{discuss2}
{\colorr 
\section{数値実験}

$(t_k^n)_k$: $[0,T]$の分割．$n_k^{c,j}$: $j$成分の$k$番目の分割内の閾値以下の増分の数, $n_k^c=\sum_jn_k^{c,j}$.

\subsection{$\sigma$の推定の高速化}

\EQQ{H^1=\sum_k(X_k^\top S_k^{-1}X_k+\log \det S_k)}
の最大化．
\EQQ{S_k=\MAT{ccc}{\Sigma_{11}G_{11} & \cdots & \Sigma_{1m}G_{1m} \\ \vdots & \ddots & \vdots \\ \Sigma_{m1}G_{m1} & \cdots & \Sigma_{mm}G_{mm}}=D_\Sigma + G_\Sigma=D_\Sigma^{1/2}(I+D_\Sigma^{-1/2}G_\Sigma D_\Sigma^{-1/2})D_\Sigma^{1/2}.}

$E_j$を$n_k^c\times n_k^{c,j}$行列で$[E_j]_{lm}=1$ if $l-m=\sum_{1\leq j'\leq j-1}n_k^{c,j'}$, $[E_j]_{lm}=0$ otherwiseとする．

\EQNN{S_k^{-1}&=D_\Sigma^{-1/2}\sum_{p=0}^\infty (D_\Sigma^{-1/2}G_\Sigma D_\Sigma^{-1/2})^pD_\Sigma^{-1/2} \\
  &=\sum_{p=0}^\infty D_\Sigma^{-1}(G_\Sigma D_\Sigma^{-1})^p \\
  &=\sum_{p=0}^\infty \sum_{\substack{1\leq i_0,\cdots, i_p\leq m \\ i_{k-1}\neq i_k}}E_{i_0}^\top \Sigma_{i_0,i_0}D_{i_0,i_0}^{-1}
  \bigg(\prod_{k=1}^p\Sigma_{i_{k-1},i_k}^{-1}G_{i_{k-1},i_k}\Sigma_{i_k,i_k}^{-1}D_{i_k.i_k}^{-1}\bigg)E_{i_p}.}

\EQQ{X_k^\top S_k^{-1}X_k=\sum_{p=0}^\infty \sum_{\substack{1\leq i_0,\cdots, i_p\leq m \\ i_{k-1}\neq i_k}} 
  \bigg(\Sigma_{i_0,i_0}^{-1}\prod_{k=1}^p\Sigma_{i_{k-1},i_k}\Sigma_{i_k,i_k}^{-1}\bigg)
  \bigg(Z_{i_0}^\top D_{i_0,i_0}^{-1} \prod_{k=1}^p G_{i_{k-1},i_k}D_{i_k,i_k}^{-1}Z_{i_k}\bigg),}
\EQNN{\log \det S_k&=\log\det D_\Sigma +\log \det (I+D_\Sigma^{-1/2}G_\Sigma D_\Sigma^{-1/2}) \\
&=\log\det D_\Sigma +\sum_{p=1}^\infty (-1)^{p-1} {\rm tr}((D_\Sigma^{-1}G_\Sigma)^p) \\
&=\log\det D_\Sigma +\sum_{p=1}^\infty (-1)^{p-1} \sum_{\substack{1\leq i_0,\cdots, i_p\leq m \\ i_{k-1}\neq i_k}} 
{\rm tr}\bigg(\prod_{k=1}^p \Sigma_{i_k,i_k}\Sigma_{i_k,i_{k+1}}D_{i_k,i_k}^{-1}G_{i_k,i_{k+1}}\bigg).}

$d=2$なら$i_0$が決まれば$i_1,\cdots, i_p$が決まる．

\subsection{$\mu$の推定}

$v_k^l(\mu_l)=(|I_i^l|)_{i;(\SIm,\SI]\subset (t_{k-1}^n,t_k^n]}$, $v_k(\mu)=(v_k^1(\mu_1)^\top ,\cdots, v_k^m(\mu_m)^\top)^\top$として
\EQQ{-\frac{1}{2}\sum_k(\Delta_k X-v_k(\mu))^\top S_k^{-1}(\Delta_k X-v_k(\mu)).}

\EQQ{\sum_j\mu_j\sum_k(v_k^j)^\top[S_k^{-1}]^{j,\cdot} X_k -\frac{1}{2}\sum_{j_1,j_2}\mu_{j_1}\mu_{j_2}\sum_k(v_k^{j_1})^\top[S_k^{-1}]^{j_1,j_2}v_k^{j_2}}
の最大化．

\subsection{ジャンプ推定}

$Y_i^j\sim N(0,v_j^2)$.

\EQQ{H^j_J(\lambda, v)=\sum_i \log f^j(\Delta_i^j X) 1_{\{|\Delta_i^j X|>c_n\}} -nh_n\lambda^j}
の最大化.

\EQQ{\log f^j(\Delta_i^j X) =\log \lambda^j-\frac{1}{2}\frac{(\Delta_i^j X)^2}{v_j^2}-\frac{1}{2}\log \det v_j^2}
なので，$CC_j=\sum_i(\Delta_i^j X)^2$, $n_J^j=\sum_i1_{\{|\Delta_i^j X|>c_n\}}$として，
\EQQ{\hat{\lambda}^j=\frac{n_J^j}{nh_n}, \quad \hat{v}_j^2=\frac{CC_j}{n_J^j}.}

}
\end{discuss2}

\end{document}